\documentclass[letter,12pt]{article}

\usepackage{graphicx}
\usepackage{amsmath,amsfonts,amssymb,mathrsfs,amsthm}
\usepackage{bbm}
\usepackage{graphicx}
\usepackage{algorithm}
\usepackage{algorithmic}
\usepackage{xcolor}
\usepackage{bm}
\newcommand{\iindex}{{\cal I}}
\newcommand{\jndex}{{\cal J}}
\newcommand{\jndexd}{{\cal J}_D}
\newcommand{\jndexn}{{\cal J}_N}
\newcommand{\ct}{T_{\cal I}}

\newcommand{\diam}{{\rm diam}}

\newcommand{\nmin}{n_{\rm min}}

\newbox\Ea
\setbox\Ea=\hbox{\raise0.9pt\hbox{$=$}}
\newcommand{\ec}{\mathrel{\hbox{$\copy\Ea\kern-\wd\Ea\raise-3.5pt\hbox{$\sim$}$}}}

\newtheorem{assumption}{Assumption}

\theoremstyle{definition} \newtheorem{definition}{Definition}
\theoremstyle{plain} \newtheorem{lemma}{Lemma}
\theoremstyle{plain} \newtheorem{theorem}{Theorem}
\theoremstyle{remark} \newtheorem{remark}{Remark}
\theoremstyle{plain} \swapnumbers %
\begin{document}

\title{A Nearly Optimal Multigrid Method for
  General Unstructured Grids\thanks{This work
    was supported by NSF DOE DE-SC0006903,  DMS-1217142 and Center for Computational Mathematics and
Applications at Penn State.}
}

\date{}
\author{Lars Grasedyck\thanks{Institut f\"ur Geometrie und Praktische Mathematik,
              RWTH Aachen, Templergraben 55, 52056 Aachen, Germany, 
              Email:\texttt{lgr@igpm.rwth-aachen.de}}
    \and Lu Wang \thanks{Department of Mathematics, The Pennsylvania State University, University Park, PA 16802. Email:
    \texttt{wang\_l@math.psu.edu}}\and Jinchao Xu
    \thanks{Department of Mathematics, The Pennsylvania State University, University Park, PA 16802. Email:
    \texttt{xu@math.psu.edu}}
}



\maketitle

\begin{abstract}
  In this paper, we develop a multigrid method on unstructured
  shape-regular grids. For a general shape-regular unstructured grid
  of ${\cal O}(N)$ elements,
  we present a construction of an auxiliary coarse grid hierarchy
  on which a geometric multigrid method can be applied together with a
  smoothing on the original grid by using the auxiliary space
  preconditioning technique.  Such a construction is realized by a
  cluster tree which can be obtained in ${\cal O}(N\log N)$ operations
  for a grid of $N$ elements. This tree structure in turn is used for
  the definition of the grid hierarchy from coarse to fine. For the
  constructed grid hierarchy we prove that the convergence rate of the
  multigrid preconditioned CG for an elliptic PDE is $1 - {\cal
    O}({1}/{\log N})$. Numerical experiments confirm the theoretical
  bounds and show that the total complexity is in ${\cal O}(N\log N)$.
%\keywords{clustering \and multigrid \and auxiliary space \and finite elements}
\end{abstract}
\vspace{1ex} 


\section{Introduction}\label{sec:intro}
We consider a sparse linear system
\begin{equation}\label{eq:lineareq}
Au =f
\end{equation}
that arises from the discretization of an elliptic partial
differential equation. In recent decades, multigrid (MG)
methods have been well established as one of the most efficient
iterative solvers for (\ref{eq:lineareq}).  Moreover, intensive
research has been done to analyze the convergence of MG. In
particular, it can be proven that the geometric multigrid (GMG) method has linear complexity ${\cal O}(N)$ in terms of computational and memory complexity for a large class of elliptic boundary value problems. 

Roughly speaking, there are two different types of theories that have
been developed for the convergence of GMG.  For the first kind theory
that makes critical use of elliptic regularity of the underlying
partial differential equations as well as approximation and inverse
properties of the discrete hierarchy of grids, we refer to Bank and
Dupont \cite{bank1981}, Braess and Hackbusch \cite{braess1983},
Hackbusch\cite{hackbusch1985}, and Bramble and
Pasciak\cite{bramble1987}.   The second kind of theory makes
minor or no elliptic regularity assumption, we refer to Yserentant
\cite{yserentant1986}, Bramble, Pasciak and Xu \cite{bramble1990},
Bramble, Pasciak, Wang and Xu \cite{bramble1991}, Xu \cite{XuThesis,Xu1992} and Yserentant \cite{yserentant1993old}, and Chen, Nochetto
and Xu \cite{xu.chen.nochetto.2012,Xu.J.Chen.L.Nochetto.R2009}.

The GMG method, however, relies on a given hierarchy of geometric
grids.  Such a hierarchy of grids is sometimes naturally available, for
example, due to an adaptive grid refinement or can be obtained in
some special cases by a coarsening algorithm \cite{chen2010}.  But in
most cases in practice, only a single (fine) unstructured grid
is given. This makes it difficult to generate a sequence of nested
meshes. To circumvent this difficulty, non-nested geometric multigrid
and relevant convergence theories have been developed.  One example of
such kind of theory is by Bramble, Pasciak and Xu
\cite{bramble1991analysis}.  In this work, optimal convergence
theories are established under the assumption that a non-nested
sequence of quasi-uniform meshes can be obtained.  Another example is
the work by Bank and Xu \cite{bank1995} that gives a nearly optimal
convergence estimate for a hierarchical basis type method for a
general shape-regular grid in two dimensions.  This theory is based on
non-nested geometric grids that have nested sets of nodal points
from different levels.

One feature in the aforementioned MG algorithms and their theories is that
the underlying multilevel finite element subspaces are not nested,
which is not always desirable from both theoretical and practical
points of view.  To avoid the non-nestedness, many different MG
techniques and theories have been explored in the literature.  One
such a theory was developed by Xu \cite{Xu1996a} for a semi-nested MG
method with an unstructured but quasi-uniform grid based on an
auxiliary grid approach.  Instead of generating a sequence of
non-nested grids from the initial grid, this method is based on a
single auxiliary structured grid whose size is comparable to the
original quasi-uniform grid.  While the auxiliary grid is
not nested with the original grid, it contains a
natural nested hierarchy of coarse grids.  Under the assumption that
the original grid is quasi-uniform, an optimal convergence theory was
developed in \cite{Xu1996a} for second order elliptic boundary
problems with Dirichlet boundary conditions.

Totally nested multigrid methods can also be obtained for general
unstructured grids, for example,  algebraic multigrid
(AMG) methods.  Most AMG methods, although their derivations are
purely algebraic in nature, can be interpreted as nested MG
when they are applied to finite element systems based on a geometric
grid. AMG methods are usually very robust and converge quickly for 
Poisson-like problems \cite{BRMCRU84,RUST85}.  There are many
different types of AMG methods: the classical AMG
\cite{Stuben1983,Brandt1984}, smoothed aggregation AMG
\cite{vanvek1992,vanvek1996,Brezina2005}, AMGe
\cite{Jones2002,Lashuk2008}, unsmoothed aggregation AMG
\cite{Bulgakov1993,Braess1995} and many others.  Highly efficient sequential
and parallel implementations are also available for both CPU and GPU systems
\cite{Henson2002,brannick2013parallel,lu2013}. AMG methods have been
demonstrated to be one of the most efficient solvers for many
practical problems \cite{samg}.  Despite of the great success in
practical applications, AMG still lacks solid theoretical
justifications for these algorithms except for two-level theories
\cite{vanvek1995,vanvek1996,stuben1999,stuben2001,falgout2005,falgout2004,brezina2012improved}.  For a
truly multilevel theory, using the theoretical framework developed in
\cite{bramble1991,Xu1992}, Van\v{e}k, Mandel, and Brezina
\cite{vanek1994} provide a theoretical bound for the smoothed
aggregation AMG under some assumption about the aggregations.  Such an 
assumption has been recently investigated in \cite{brezina2012improved} for
aggregations that are controlled by auxiliary grids that are similar
to those used in \cite{Xu1996a}.

The aim of this paper is to extend the algorithm and theory in Xu
\cite{Xu1996a} to shape regular grids that are not necessarily
quasi-uniform.  The lack of quasi-uniformity of the original grid
makes the extension nontrivial for both the algorithm and the theory.
First, it is difficult to construct auxiliary hierarchical grids
without increasing the grid complexity, especially for grids on complicated
domains. The way we construct the hierarchical structure is to
generate a cluster tree, based on the geometric information of the
original grid~\cite{Finkel1974,GRHALE01,GRHALE03,feuchter2003}.
This auxiliary cluster tree has also been used as a coarsening process of
the UA-AMG \cite{lu2013}.   Secondly, it is also not straightforward
to establish optimal convergence for the geometric multigrid applied
to hierarchy of auxiliary grids that can be highly locally
refined. 

The rest of the paper is organized as follows.  In \S
\ref{sec:preliminaries}, we discuss some basic assumptions on the given
triangulation and review multigrid theories and the
auxiliary space method. An abstract analysis is provided based on four assumptions.
In \S \ref{sec:construct} we introduce the detailed construction of the structured auxiliary space by an auxiliary cluster tree and an improved treatment of the boundary
region for Neumann boundary conditions.  In \S \ref{sec:proofaux}, we describe the auxiliary space multgrid preconditioner (ASMG) and estimate the condition number by verifying the assumptions of the abstract theory.
Finally, in \S \ref{sec:numerics}, we provide some numerical examples to verify our theory.

For simplicity of exposition, the algorithm and theory are presented
in this paper for the two dimensional case by using a quadtree. They can
be generalized for the three dimensional case without intrinsic
difficulties by using an octree. 

\section{Preliminaries}\label{sec:preliminaries}
We present a multigrid method for second order elliptic problems on a
complicated domain $\Omega\subset\mathbb{R}^d$ ($1\le d\le 3$). Let
$\mathcal{V}$ be the initial Hilbert space with inner product
$a(\cdot,\cdot)$ and energy norm $\parallel\cdot\parallel_{A}$.  The
variational problem we want to solve is
\begin{equation}\label{eq:bilinearform}
Find\ v\in \mathcal{V}\ s.t.\ \ a(u,v)  := \langle Au,v\rangle = \langle f,v\rangle \qquad \forall v\in \mathcal{V}
\end{equation}
In order
to simplify the presentation, we restrict ourselves to the Poisson problem 
discretized by piecewise linear 
finite elements.    Assume that the nodal basis functions of
  $\mathcal{V}$ are $\{\varphi_i\}_{i\in\jndex}$ and the local elements supporting patch of $\varphi_i$ is
  $\omega_{i}$ where $\jndex$ is the node index set. For any $v\in \mathcal{V}$, there exists $\xi\in \mathbb{R}^{N}$
  such that $v=\sum_{i\in\jndex}\xi_i\varphi_i$. In this case, the
system matrix $A$ has entries of the form
\[a_{i,j} = \int_{\Omega}(\nabla\varphi_i(x),\nabla\varphi_j(x)) dx\]
with basis functions $\varphi_i(x)$ that are continuous and  piecewise affine on triangles 
$\tau_\nu, \nu\in\iindex:=\{1,\ldots,N\}$ of the triangulation $\mathcal{T}$ of the 
polygonal and connected domain $\Omega\subset\mathbb{R}^d$,
\[\overline\Omega = \bigcup_{\nu\in\iindex}\overline{\tau_\nu}\]

\subsection{Properties of the triangulation}
The triangulation is assumed to be conforming and shape-regular in the
sense that the ratio of the circumcircle and inscribed circle is
bounded uniformly \cite{Ciarlet}, and it is a K-mesh in the sense that the ratio of diameters between neighboring elements is bounded uniformly. All elements $\tau_i$ are assumed to be shape-regular but
not necessarily quasi-uniform, so the diameters can vary globally and
allow a strong local refinement.

The vertices of the triangulation are denoted by $(p_j)_{j\in\jndex}$. Some of the vertices
are Dirichlet nodes, $\jndexd\subset\jndex$, where we impose essential Dirichlet boundary 
conditions, and some are Neumann vertices, $\jndexn\subset\jndex$, where we impose natural Neumann boundary conditions.

The following construction will be given for the case $d=2$, but
a generalisation to $d>2$ is straightforward . 
%% (maybe a future article and a few sentences in the outlook).

For each of the triangles $\tau_i\in\mathcal{T}$, we use the barycenter
\[ \xi_i := \frac{p_1(\tau_i) + p_2(\tau_i) + p_3(\tau_i)}{3},\]
where $p_1(\tau_i), p_2(\tau_i), p_3(\tau_i)\in\mathbb{R}^2$ 
are the three vertices of the triangle $\tau_i$, as in Figure \ref{tri_fig}.

\noindent{\bf Notation: }
  We denote the minimal distance between the triangle barycenters of the 
  grid by 
  \[ h := \min_{i,j\in\iindex}\|\xi_i-\xi_j\|_2.\]
  The diameter of $\Omega$ is denoted by 
  \[H := \max_{x,y\in\Omega}\|x-y\|_2.\]

\begin{figure}[!htb]
  \begin{center}
    \includegraphics[width=4.5cm]{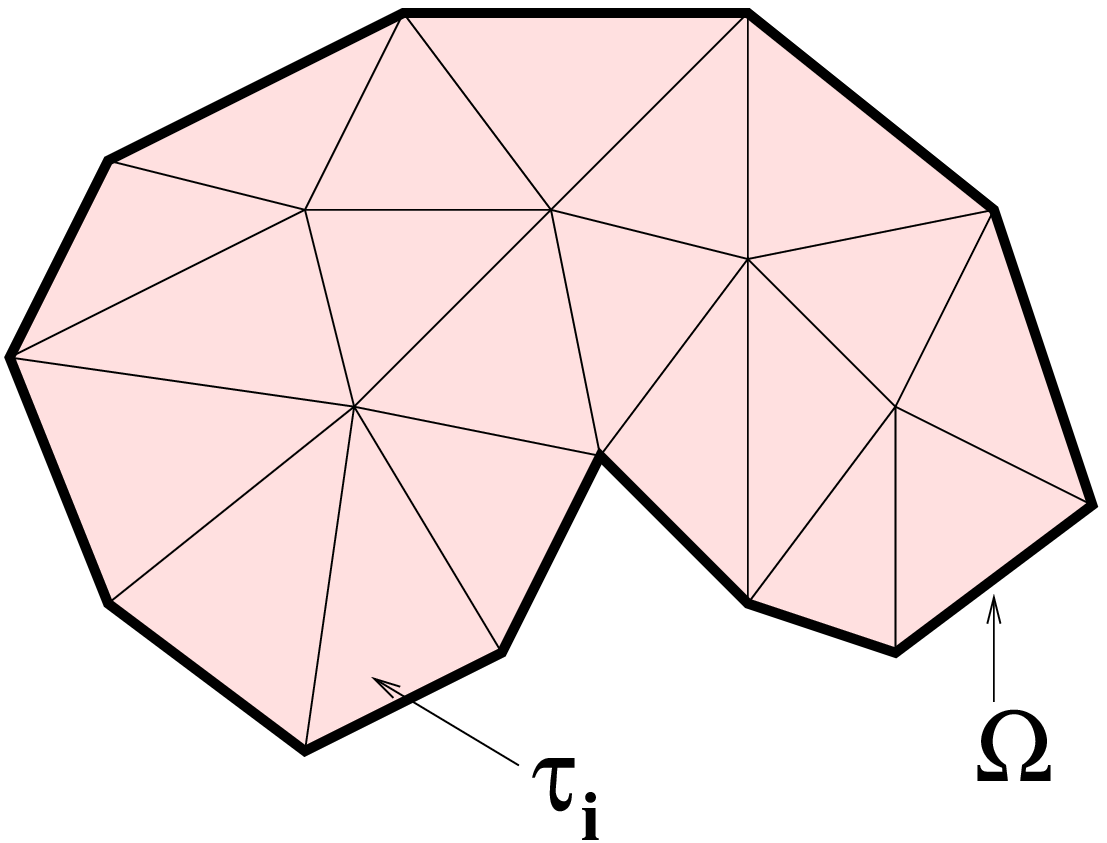}
    \hspace{2cm}
    \includegraphics[width=4.5cm]{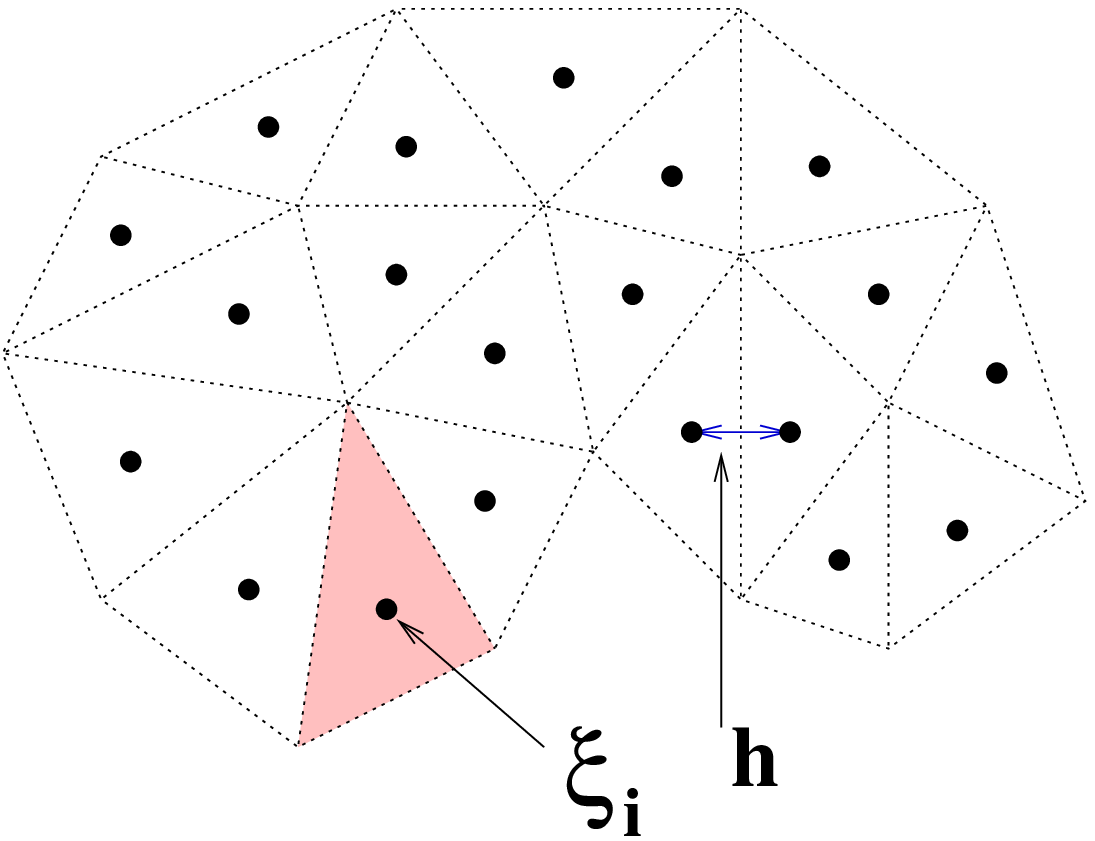}
  \end{center}
  \caption{Left: The triangulation $\mathcal{T}$ of $\Omega$ 
    with elements $\tau_i$.
    Right: The barycenters $\xi_i$ (dots) and the minimal 
    distance $h$ between barycenters.\label{tri_fig}}
\end{figure}

In order to prove the desired nearly linear complexity
estimate, we have to assume that the refinement level
of the grid is algebraically bounded in the following sense.

\begin{assumption}\label{def_p}
  We assume that 
  $H/h \ec N^q$ for a small number $q$, e.g. $q=2$.
\end{assumption}

The above assumption allows
an algebraic grading towards a point but it forbids a
geometric grading. The assumption is sufficient but not necessary: 
the construction of the auxiliary grids might still be
of complexity ${\cal O}(N \log N)$ or less if Assumption
\ref{def_p} is not valid, but it would
require more technical assumptions in order to prove this. 

\subsection{Auxiliary space preconditioning theory}
The auxiliary space method, developed in
\cite{nepomnyaschikh1992,Xu1996a}, is for designing preconditioners by
using the auxiliary spaces which are not necessarily subspaces of the
original space.  Here, the original space is the finite element space
$\mathcal{V}$ for the given grid $\mathcal T$ and the preconditioner
is the multigrid method on the sequence $(V_\ell)_{\ell=1}^{J}$ of FE
spaces for the auxiliary grids $({\mathcal T}_\ell)_{\ell=1}^{J}$.

The idea of the method is to generate an auxiliary space $V$ with inner product 
$\tilde{a}(\cdot,\cdot) = \langle\tilde{A}\cdot,\cdot\rangle$ and  energy norm 
$\parallel\cdot\parallel_{\tilde{A}}$. Between the spaces there is a suitable linear transfer 
operator $\Pi:V\mapsto \mathcal{V}$, which is continuous and surjective. 
$\Pi^{t}:\mathcal{V}\mapsto V$ is the dual operator of $\Pi$ in the default inner products
\[
  \langle\Pi^{t}u,\tilde{v}\rangle = \langle u,\Pi\tilde{v}\rangle,\quad 
  \text{for all }u\in \mathcal{V}, \tilde{v}\in V.
  \]
% $\Pi^{*}:V\mapsto \tilde{V}$ is the joint operator of $\Pi$ 

%{\color{red}{\bf[I have modified the notation a little, on the auxiliary it does not have to solve it exactly]}

In order to solve the linear system $Ax=b$, we require a preconditioner B defined by
\begin{equation}\label{eq:fasp}
B := S + \Pi \tilde{B} \Pi^{t},
\end{equation}
where $S$ is the smoother and $\tilde{B}$ is the preconditioner of $\tilde{A}$. 
 
The estimate of the condition number $\kappa(BA)$ is given below.
\begin{theorem}\label{thm:FASP}
Assume that there are nonnegative constants $c_{0},c_{1}$, and $c_{s}$, such that
\begin{enumerate}
\item the smoother $S$ is bounded in the sense that
\begin{equation}\label{eqn:smoother}
\quad  \|v\|_A\leq c_{s}\|v\|_{S^{-1}} \qquad\forall v\in \mathcal{V},
\end{equation}
\item the transfer operator $\Pi$ is bounded, 
\begin{equation}\label{eqn:transfer}
\| \Pi w\|_{A} \leq c_1 \|w\|_{\tilde{A}} \qquad \forall w\in V, 
\end{equation}
\item the transfer is stable, i.e. for all $v\in \mathcal{V}$ there exists $v_0\in \mathcal{V}$ and $w\in V$ such that
\begin{equation}\label{eqn:transfer_stable}
\quad v = v_0 + \Pi w \quad\text{ and }\quad \|v_0\|_{S^{-1}} + \| w\|_{\tilde{A}} \leq c_0 \| v\|_{A},
\quad\mbox{and}
\end{equation}
\item the preconditioner $\tilde{B}$ on the auxiliary space is optimal, i.e. for any $\tilde{v}\in V$, there exists $m_{1}>m_{0}>0$, such that
\begin{equation}\label{eqn:optimal-auxsolver}
m_{0}\|\tilde{v}\|^{2}_{\tilde{A}}\leq (\tilde{B}\tilde{A}\tilde{v},\tilde{v})\leq m_{1}\|\tilde{v}\|^{2}_{\tilde{A}}.
\end{equation}
\end{enumerate}
Then, the condition number of the preconditioned system defined by (\ref{eq:fasp}) can be bounded by
\begin{equation}
\kappa(BA) \leq \frac{m_{1}}{m_{0}}c_{0}^{2}(c_{s}^{2} + c_{1}^{2}).
\end{equation}
\end{theorem}

We also can combine the smoother S and the auxiliary grid correction multiplicatively with a 
preconditioner $B$ in the form \cite{huwuxuzheng2013,huxuzhang2013}
\begin{equation}\label{eq:fasp2}
I-B_{co}A = (I-S^TA)(I-BA)(I-SA)
\end{equation}
which leads to Algorithm \ref{Alg:fasp}. The combined preconditioner, under suitable scaling assumptions 
performs no worse than its components.
\begin{algorithm}[htb]
  \caption{Multiplicative Auxiliary Space Iteration Step\label{Alg:fasp}}
\begin{algorithmic}
  \STATE Given S, $B$, and an initial iterate $u^{k,0}:=u^k$
  \STATE (1) $u^{k,1} := u^{k,0} - S(b - Au^{k,0})$;
  \STATE (2) $u^{k,2} := u^{k,1} - B(b - Au^{k,1})$;
  \STATE (3) $u^{k+1} := u^{k,2} - S^T(b - Au^{k,2})$;
  \RETURN Improved approximate solution $u^{k+1}$.
\end{algorithmic}
\end{algorithm}
\begin{theorem}
  Suppose there exists $\rho\in[0,1)$ such that for all $v\in V$, we have
  \[
  \|(I-SA)v\|_A^2\leq \rho\|v\|^2_A,
  \]
 then the multiplicative 
  preconditioner $B_{co}$ yields the bound 
  \begin{equation}
    \kappa(B_{co}A)\leq \frac{(1-m_1)(1-\rho)+m_1}{(1-m_0)(1-\rho)+m_0},
  \end{equation}
  for the condition number, and
  \begin{equation}
    \kappa(B_{co}A)\leq \kappa(BA).
  \end{equation}
\end{theorem}
%}
According to Theorem 1 and 2, our goal is to construct an auxiliary space $V$ in which we are able to
define an efficient preconditioner. The preconditioner will be the geometric multigrid method on a
suitably chosen hierarchy of auxiliary grids. 
Additionally, the space has to be close enough to $\mathcal{V}$ so that the transfer from $\mathcal{V}$ to $V$ fulfils 
(\ref{eqn:transfer}) and (\ref{eqn:transfer_stable}). This goal is achieved by a density fitting of the 
finest auxiliary grid $\mathcal{T}_{J}$ to $\mathcal{T}$. In order to prove (\ref{eqn:optimal-auxsolver}), we use the multigrid theory for the auxiliary grids $\{\mathcal{T}_{\ell}\}_{\ell=1}^{J}$ from the viewpoint of the method of subspace corrections.

\subsection{An abstract multigrid theory}
In the spirit of divide and conquer, we can decompose any space $V$ as the summation of subspaces $V = \sum_{i=0}^{L}V_{i},\ V_{i}\subset V$. Since $\sum_{i=0}^{L}V_{i}$ may not be a direct sum, the decomposition $u = \sum_{i=0}^{L}u_{i}$ is not necessarily unique for $u\in V$.

We will use the following operators, for $i = 1,\ldots,L$:
\begin{itemize}
\item $Q_i:V\rightarrow V_i\quad$ the projection in the $L_{2}$ inner product $(\cdot,\cdot)$;
\item $I_i:V_i\rightarrow V\quad$ the natural inclusion to $V$;
\item $P_i:V\rightarrow V_i\quad$ the projection in the inner product $(\cdot,\cdot)_{A}$;
\item $A_{i}:V_{i}\rightarrow V_{i}\quad$ the restriction of A to the subspace $V_{i}$;
\item $R_i:V_i\rightarrow V_i\quad$ an approximation of $(A_{i})^{-1}$ which means the smoother;
\item $T_{i}:V\rightarrow V_{i}\quad$ $T_{i}=R_{i}Q_{i}A=R_{i}A_{i}P_{i}$.
\end{itemize}
For any $u\in V$ and $u_i,v_i\in V_i$, these operators fulfil the trivial equalities 
\[
( Q_iu,v_i)   = ( u,I_iv_i) = ( I_i^t u, v_i),
\]
\[
( A_{i} P_iu,v_i) = a(u,v_i)  =( Q_iA u,v_i),
\]
\[
( A_{i}u_i,v_i)  = a(u_i,v_i) 
= (A u_i,v_i) =( Q_iA I_i u_i,v_i).
\]
Assume we know the value of $u^{k}$, if we perform the subspace correction in a successive way, it reads in operator form as 
\begin{align*}
v^{0} &= u^{k},\\
\quad v^{i+1} &= v^{i} + I_{i}R_{i}Q_{i}(f - Av^{i}), i=1,\cdots,L,\\ \quad u^{k+1} &= v^{L+1}
\end{align*}

The corresponding error equation is 
\[
(I-BA)u = u - u^{k+1} =  \left[\prod_{i=1}^{L} (I - I_iR_iQ_iA)\right](u-u^{k}).
\]

Suppose that we have nested finite element spaces: $V_0\subset V_1\subset V_2\subset\cdots\subset V_J=V$ and $\{\varphi_{\ell,i}\}$ as the basis functions of $V_{\ell}$. Define $V_{k,i}:=span\{\varphi_{k,i}\}$. Then, $\sum_{\ell,i}V_{\ell,i}$ is the decomposition of $V_{J}$. Then given $v\in V_{J}$, we have the decomposition $v=\sum_{\ell,j}v_{\ell,j}$. Similarly define $P_{\ell,i}$ as the projection from $V_{J}$ to $V_{\ell,i}$. In this case, the successive subspace correction method for this
decomposition is nothing but the simple Gauss-Seidel iteration. 
This algorithm is sometimes called the backslash ($\backslash$) cycle.
A V-cycle algorithm is obtained from the backslash cycle by performing
more smoothings after the coarse grid corrections.  Such an algorithm,
roughly speaking, is like a backslash ($\backslash$) cycle plus a slash
($\slash$) (a reversed backslash) cycle. It is simple to show that the convergence of the V-cycle is a consequence of the convergence of the backslash cycle. The detailed algorithm is
given in Algorithm \ref{Alg:mg}.

\begin{algorithm}
  \caption{Geometric Multigrid Method\label{Alg:mg}}
 \begin{algorithmic}
\STATE For $\ell=0$, define $B_0=A_0^{-1}$.  Assume that $B_{\ell-1}: V_{\ell-1}\rightarrow V_{\ell-1}$ is defined.  We shall now define $B_{\ell}: V_{\ell}\rightarrow V_{\ell}$ which is an iterator for the equation of the form $$A_\ell u=f.$$
\STATE {\bf pre-smoothing:} For $u^0=0$ and $k=1,2,\cdots, \nu$
        $$u^k=u^{k-1}+R_{\ell}(f-A_\ell u^{k-1})$$
\STATE {\bf Coarse grid correction:} $e_{\ell-1}\in V_{\ell-1}$ is the approximate solution of the residual equation $A_{\ell-1}e=Q_{\ell-1}(f-A_\ell u^{\nu})$ by the iterator $B_{\ell-1}$:
$$u^{\nu+1}=u^{\nu}+e_{\ell-1}=u^{\nu}+B_{\ell-1}Q_{\ell-1}(g-Au^{\nu}).$$
\STATE {\bf post-smoothing:} For $k=\nu+2,2,\cdots, 2\nu$
$$u^{k}=u^{k-1}+R_{\ell}(f-A_{\ell}u^{k-1})$$
\end{algorithmic}
\end{algorithm}

Now we present a convergence analysis based on three assumptions.
\begin{description}
\item[({\bf T) Contraction of Subspace Error Operator: }] There exists $\rho<1$ such that
  \[
  \|I-T_i\|_{A_{i}}\leq\rho\qquad \text{for all } i = 1,\ldots,L.
  \]
\item[{(\bf A1) Stable Decomposition: }] For any $v\in V$, there exists a decomposition 
  \[ v=\sum_{i=1}^{L}v_i,\quad v_i\in V_i, i=1,\ldots,L,\qquad \text{such that }\,
  \sum_{i=1}^{L}\|v_i\|^2_{A_{i}} \leq K_1 \|v\|^2_A.
  \]
\item[{(\bf A2) Strengthened Cauchy-Schwarz (SCS) Inequality: }] For any $u_i, v_i\in V_i, i=1,\ldots,L$
  \[
  \left\vert\sum_{i=1}^{L}\sum_{j=i+1}^{L}( u_i,v_j)_A\right\vert
  \leq K_2\left( \sum_{i=1}^{L}\|u_i\|^2_A \right)^{1/2}
  \left( \sum_{j=1}^{L}\|v_j\|^2_A \right)^{1/2}.
  \]
\end{description}

The convergence theory of the method is as follows.
\begin{theorem}\label{thm:ssc}
  Let $V=\sum_{i=1}^{L}V_i$ be a decomposition satisfying assumptions {\bf (A1)} and {\bf (A2)}, 
  and let the subspace smoothers $R_i$ satisfy {\bf (T)}.  Then 
  \[
  \left\|\prod_{i=1}^{L}(I-I_iR_iQ_iA)\right\|^2_{A} \leq 1 - \frac{1-\rho^2}{2K_1(1+(1+\rho)^2K_2^2)}.
  \]
\end{theorem}
The proof can be found from \cite{Xu.J.Chen.L.Nochetto.R2009,xu.chen.nochetto.2012}, which is simplified by using the XZ identity \cite{xu2002method}.

\section{Construction of the auxiliary grid-hierarchy} \label{sec:construct}

In this section, we explain how to generate a hierarchy of auxiliary grids based on the given 
(unstructured) grid $\mathcal{T}$.
The idea is to analyse and split the element barycenters by their geometric position regardless of the 
initial grid structure. Our aim is to obtain a structured hierarchy of grids that preserves some properties 
of the initial grid, e.g. the local mesh size. A similar idea has already been applied in 
\cite{Wittum2003,Sauter07,lu2013}.

\subsection{Clustering and auxiliary box-trees}

We build an auxiliary tree structure by a  geometrically regular subdivision of boxes.
For the initial step we choose a (minimal) {\bf square} bounding box of the domain $\Omega$:
\begin{equation*}
  \mathcal{B}^{1}:= [ a_1, b_1 ) \times [ a_{2}, b_{2})
      \supset \Omega, \qquad |b_1-a_1|=|b_2-a_2|.
\end{equation*}
Define the level of $\mathcal{B}^{1}$ to be $g(\mathcal{B}^{1}) = 1$. Then we subdivide $\mathcal{B}^{1}$ regularly, thus
obtaining four children $\mathcal{B}^{2}_{1},\mathcal{B}^{2}_{2}, 
\mathcal{B}^{2}_{3},\mathcal{B}^{2}_{4}$:
\begin{eqnarray*}
  \mathcal{B}^{2}_{2} 
  =
  [ a_{1}, b_{1}') \times [ a_{2}', b_{2}), 
  &&
  \mathcal{B}^{2}_{3} 
  =
  [ a_{1}', b_{1}) \times [ a_{2}', b_{2}), 
  \\
  \mathcal{B}^{2}_{1} 
  = 
  [ a_{1}, b_{1}') \times [ a_{2}, b_{2}'), 
  && 
  \mathcal{B}^{2}_{4} 
  =
  [ a_{1}', b_{1}) \times [ a_{2}, b_{2}'), 
\end{eqnarray*}
where $a_{1}^{\prime} = b_{1}^{\prime}:=(a_{1}+b_{1})/2$ and
$a_{2}^{\prime} = b_{2}^{\prime}:=(a_{2}+b_{2})/2$.  The level of $
\mathcal{B}^{2}_{i}$ is
$g(\mathcal{B}^{2}_{i})=g(\mathcal{B}^{1})+1=2$, where $i = 1,2,3,4$.
Finally, we apply the same subdivision process recursively, starting
with $\mathcal{B}^{2}_{1},\ldots,\mathcal{B}^{2}_{4}$ and define the
level of the boxes $\mathcal{B}^{\ell}_{i}$ recursively (cf. Figure
\ref{box_cluster}). This yields an infinite tree $T_{\mathrm{box}}$
with root $\mathcal{B}^{1}$.  Letting $\mathcal{B}^{\ell}_{j}$ denote
a box in this tree, we can define the cluster $t$, which is a subset
of $\iindex$, by
\begin{equation*}
  t^{\ell}_j := t(\mathcal{B}^{\ell}_{j}):=\{i\in \mathcal{I}\mid \xi _{i}\in \mathcal{B}^{\ell}_{j}\}.
\end{equation*}
This yields an infinite cluster tree with root $t(\mathcal{B}^{1})$. 
We construct a finite cluster tree $\ct$ by not subdividing nodes which are below a minimal 
cardinality $\nmin$, e.g. $\nmin:=3$. Define the nodes which have no child nodes as the leaf nodes. 
%Denote the leaf nodes as $\{\tilde{\mathcal{B}}^{\ell}_{j}\}$.
The cardinality $\#t^{\ell}_j=\#t(\mathcal{B}^{\ell}_{j})$ is the number of the barycenters in $\mathcal{B}^{\ell}_{j}$. Leaves of the cluster tree contain at most $\nmin$ indices.  For any leaf node, its parent node contains at least 4 barycenters, then the total number of leaf nodes is bounded by the number of barycenters $N$.

\begin{figure}[htb]
\begin{center}
  \includegraphics[height=2.5cm]{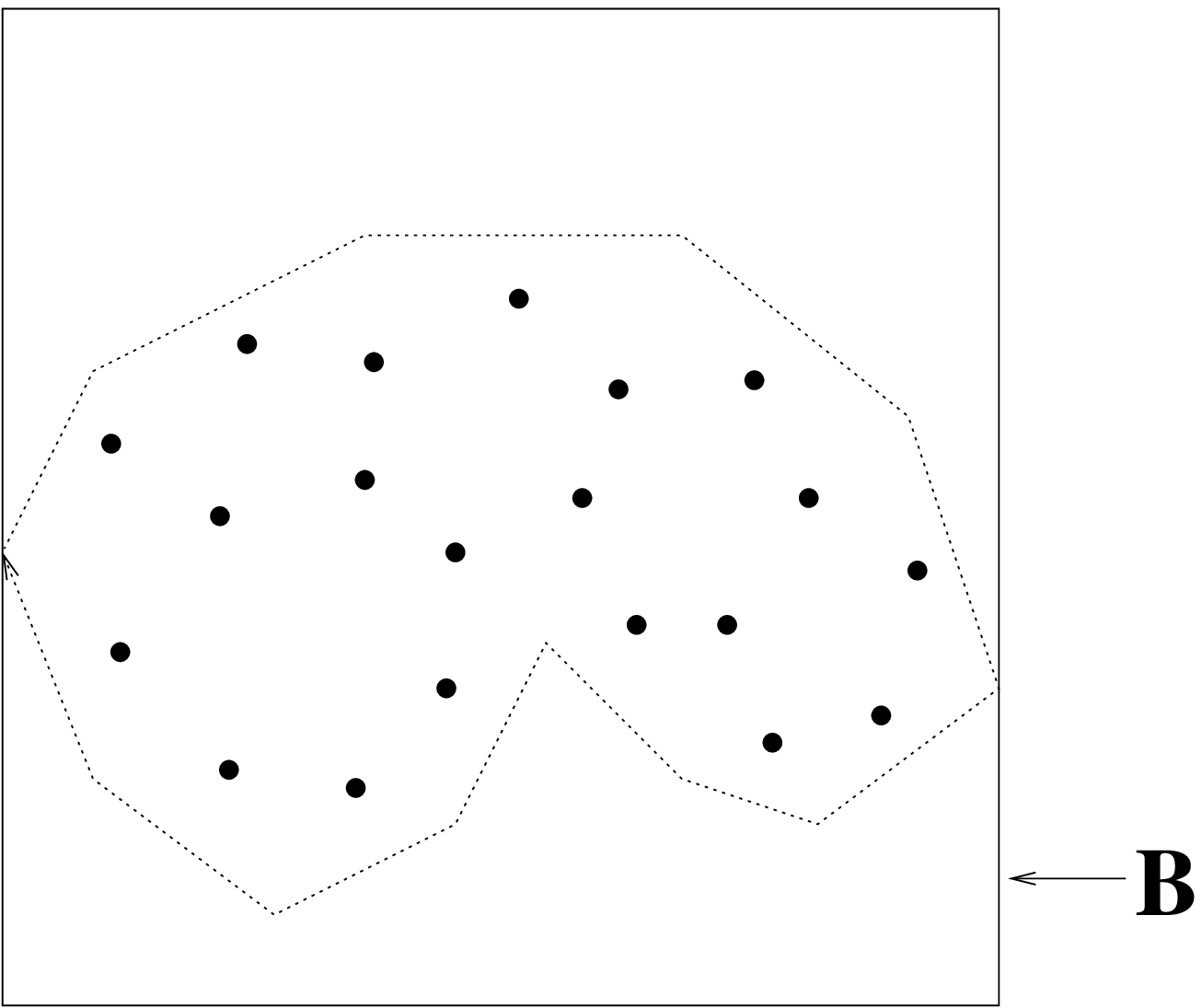}
  \hspace{0.1cm}
  \includegraphics[height=2.5cm]{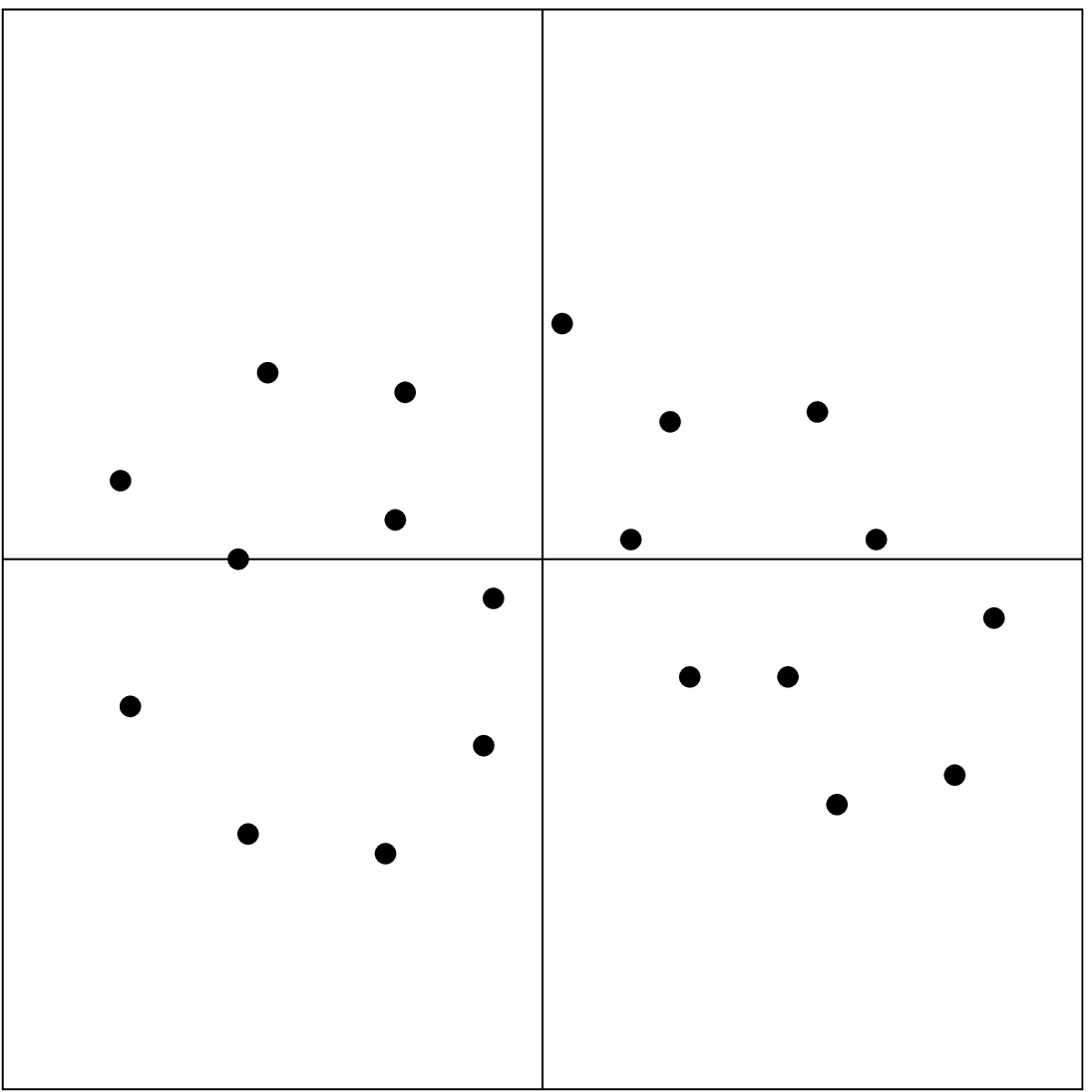}
  \hspace{0.2cm}
  \includegraphics[height=2.5cm]{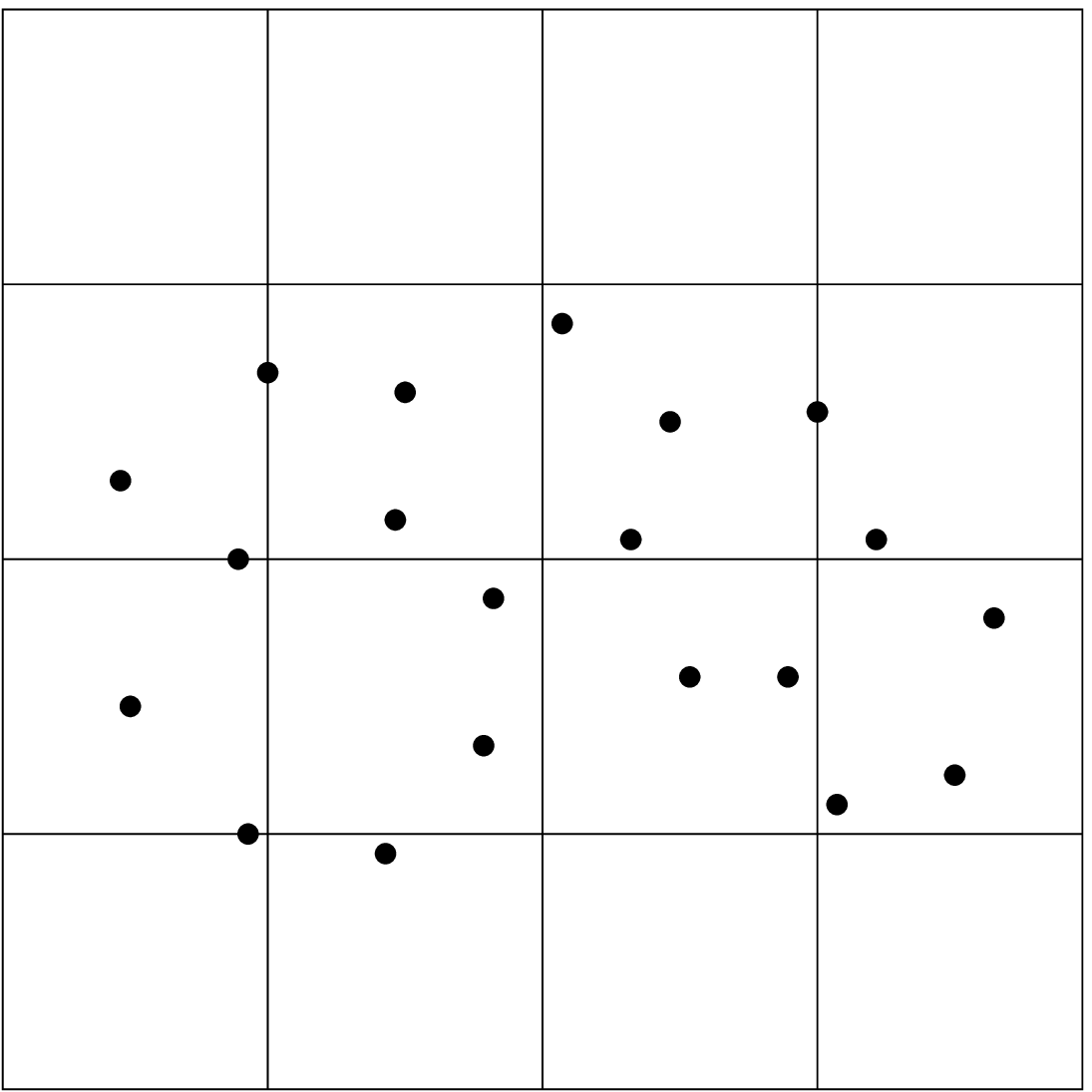}
  \hspace{0.2cm}
  \includegraphics[height=2.5cm]{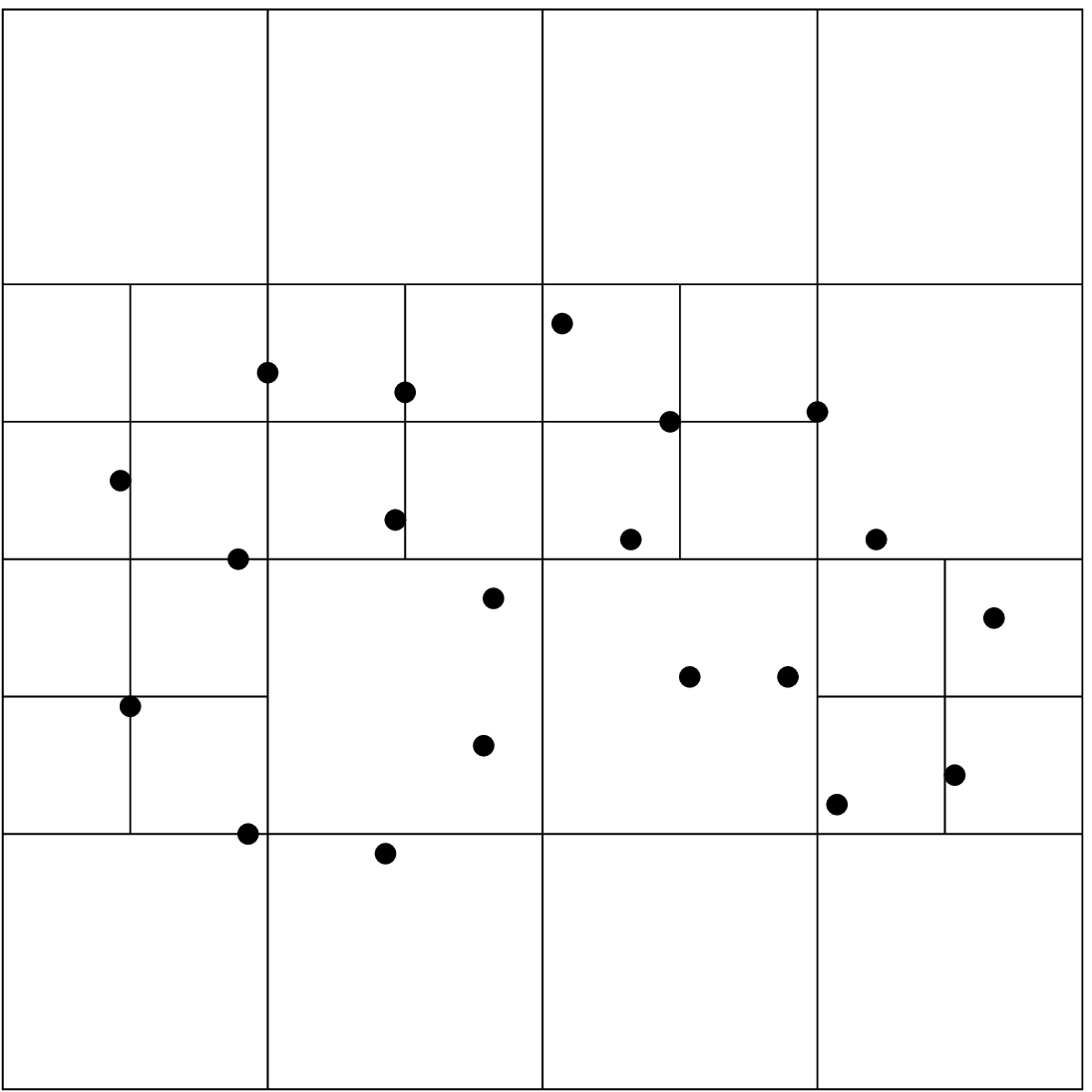}
\end{center}
\caption{Tree of regular boxes with root $\mathcal{B}^{1}$. 
  The black dots mark the corresponding barycenters $\xi_i$ 
  of the triangles $\tau_i$. Boxes with less than three
  points $\xi_i$ are leaves.
  \label{box_cluster}}
\end{figure}

\begin{remark}
  The size of a leaf box $\mathcal{B}_{j}$ can be much larger than the
  size of triangles $\tau_j$ that intersect with $\mathcal{B}_{j}$
  since a large box $\mathcal{B}_{j}$ may only intersect
  with one very small element and will not be further
  subdivided.
\end{remark}

\begin{lemma}
Suppose Assumption \ref{def_p} holds, the complexity for the construction of $\ct$ is ${\cal O}(qN \log N)$.
\end{lemma}
\begin{proof}
  First, we estimate the depth of the cluster tree. Let 
  $t=t({\cal B}_\nu)\in\ct$ be
  a node of the cluster tree and $\#t>\nmin$. By definition the distance
  between two nodes $\xi_i,\xi_j\in t$ is at least 
  \[\|\xi_i - \xi_j\|_2 \ge h.\]
  Therefore, the box ${\cal B}_\nu$ has a diameter of at least
  $h$. After each subdivision step the diameter of the boxes
  is exactly halved. Let $\ell$ denote the number of subdivisions
  after which  ${\cal B}_\nu$ was created. Then
  \[\diam({\cal B}_\nu) = 2^{-\ell} \diam({\cal B}^{1}).\]
  Consequently, we obtain
  \[h \le \diam({\cal B}_\nu) = 2^{-\ell} \diam({\cal B}^{1}) \le
  2^{-\ell} \sqrt{2} H\]
  so that by Assumption \ref{def_p},
  \[\ell \lesssim \log(H/h)\ec 
  q \log N.\]
  Therefore the depth of $\ct$ is in ${\cal O}(q\log N)$.

  Next, we estimate the complexity for the construction of $\ct$.
  The subdivision of a single node $t\in\ct$ and corresponding 
  box ${\cal B}_\nu$ is of complexity $\#t$. 
  On each level of the tree $\ct$, the nodes are disjoint, so that
  the subdivision of all nodes on one level is of complexity
  at most ${\cal O}(N)$. For all levels this sums up to at most
  ${\cal O}(q N \log N)$.
\end{proof}

\begin{remark}
  The boxes used in the clustering can be replaced by arbitrary
  shaped elements, e.g. triangles/tetrahedra or anisotropic elements --- 
  depending on the application or operator at hand. For ease of
  presentation we restrict ourselves to the case of boxes. 
\end{remark}

\begin{remark}
  The complexity of the 
  construction can also be bounded from below by ${\cal O}(N\log N)$, as is the case
  for a uniform (structured) grid. However, this complexity arises only in the 
  construction step and this step will typically be of negligible complexity. 
\end{remark}

%For each cluster $t_\nu\subset\ct$ we have an associated
%box $\mathcal{B}_\nu$. The depth of the cluster tree is 
%$p:={\rm depth}(\ct)$.
Notice that the tree of boxes is \emph{not} the
regular grid that we need for the multigrid method. A further
refinement as well as deletion of elements is necessary.

\subsection{Closure of the auxiliary box-tree}

The hierarchy of box-meshes from Figure \ref{box_cluster} is exactly 
what we want to construct: each box has at most one hanging node
per edge, namely, the fineness of two neighbouring boxes differs
by at most one level. In general this is not fulfilled.
%Due to the assumption that the initial mesh is a shape-regular mesh, we can
%only guarantee that there are at most ${\cal O}(\log N)$ hanging nodes
%per edge (see Figure \ref{box_hanging}). --- this is wrong.

We construct the grid hierarchy of nested uniform meshes starting 
from a coarse mesh $\sigma^{(0)}$ consisting of only a single box
$\mathcal{B}^{1} = 
[ a_{1}, b_{1}) \times [ a_{2}, b_{2})$,
the root of the box tree. All boxes in the meshes 
$\sigma^{(1)},\ldots,\sigma^{(J)}$ to be constructed will either 
correspond to a cluster $t$ in the cluster tree or will be 
created by refinement of a box that corresponds to a leaf of
the cluster tree. 

Let $\ell\in\{1,\ldots,J\}$ be a level that is already constructed
(the trivial start $\ell=1$ of the induction is given above).

%\begin{center}
%  \fbox{Part I: Mark elements for refinement}
%\end{center}
We mark all elements of the mesh which are then refined regularly.
Let $\mathcal{B}^{\ell}_\nu$ be an arbitrary box in $\sigma^{(\ell)}$.
The box $\mathcal{B}^{\ell}_\nu$ corresponds to a cluster $t_\nu=t(\mathcal{B}^{\ell}_\nu)\in\ct$.
The following two situations can occur:
\begin{list}{Closure}{}
\item[1. {\bf(Mark)} ]
  If $\#t_\nu > \nmin$ then $\mathcal{B}^{\ell}_\nu$ is marked for refinement.
\item[2. {\bf(Retain)} ]
  If $\#t_\nu \le \nmin$, e.g. $t_\nu=\emptyset$, then $\mathcal{B}^{\ell}_\nu$ is 
  not marked in this step. 
\end{list}

\begin{figure}[htb]
\begin{center}
  \includegraphics[height=2.2cm]{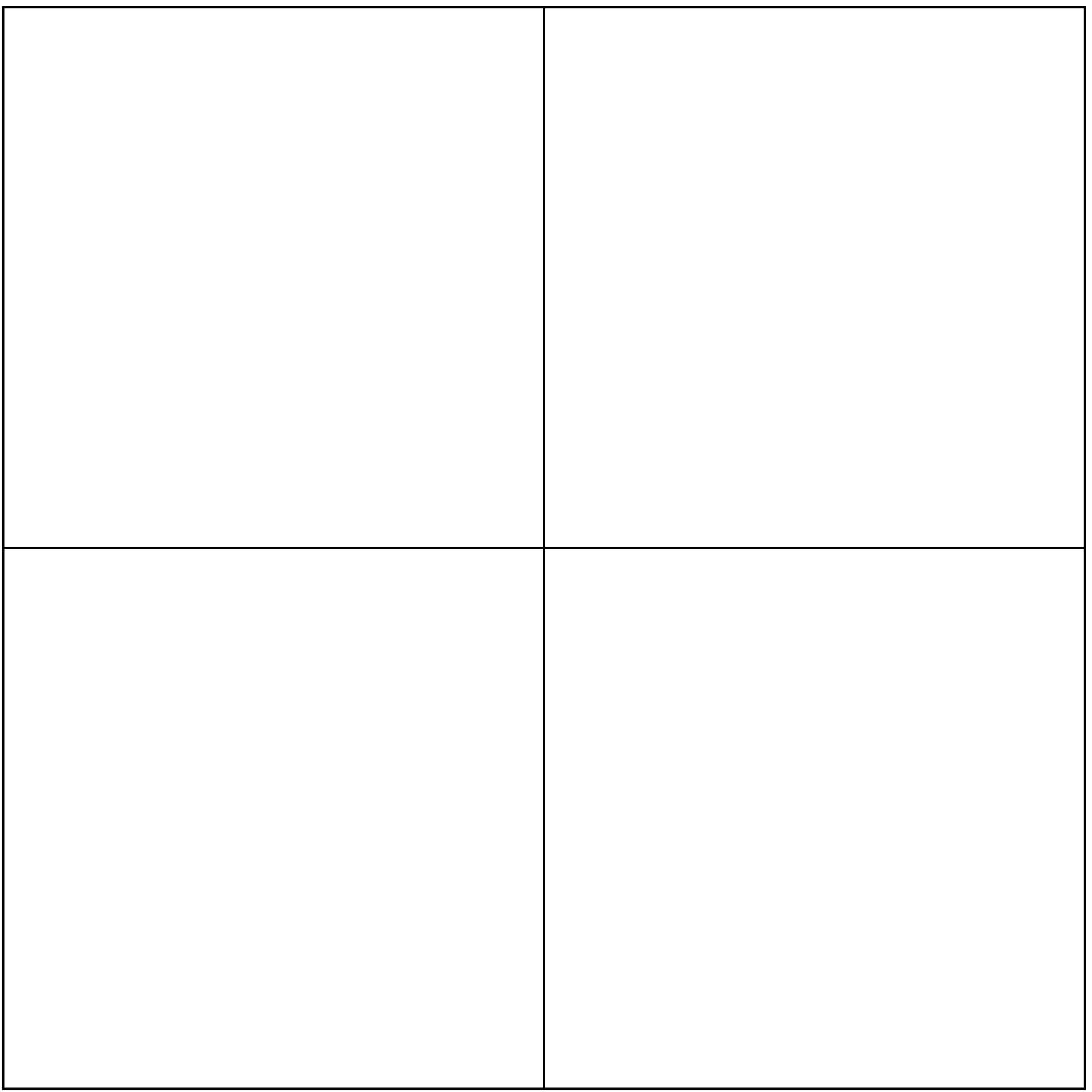}
  \hspace{0.1cm}
  \includegraphics[height=2.2cm]{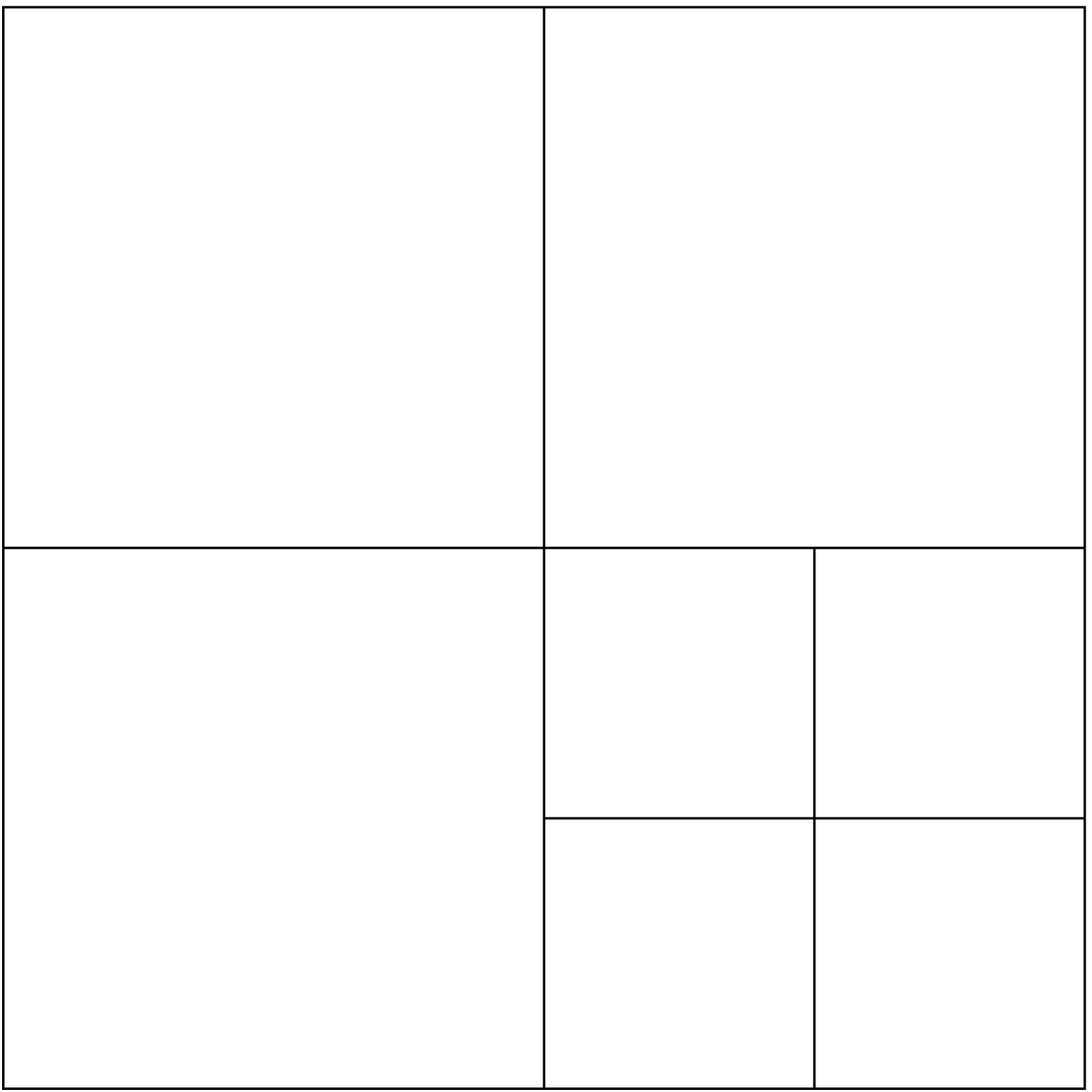}
  \hspace{0.1cm}
  \includegraphics[height=2.2cm]{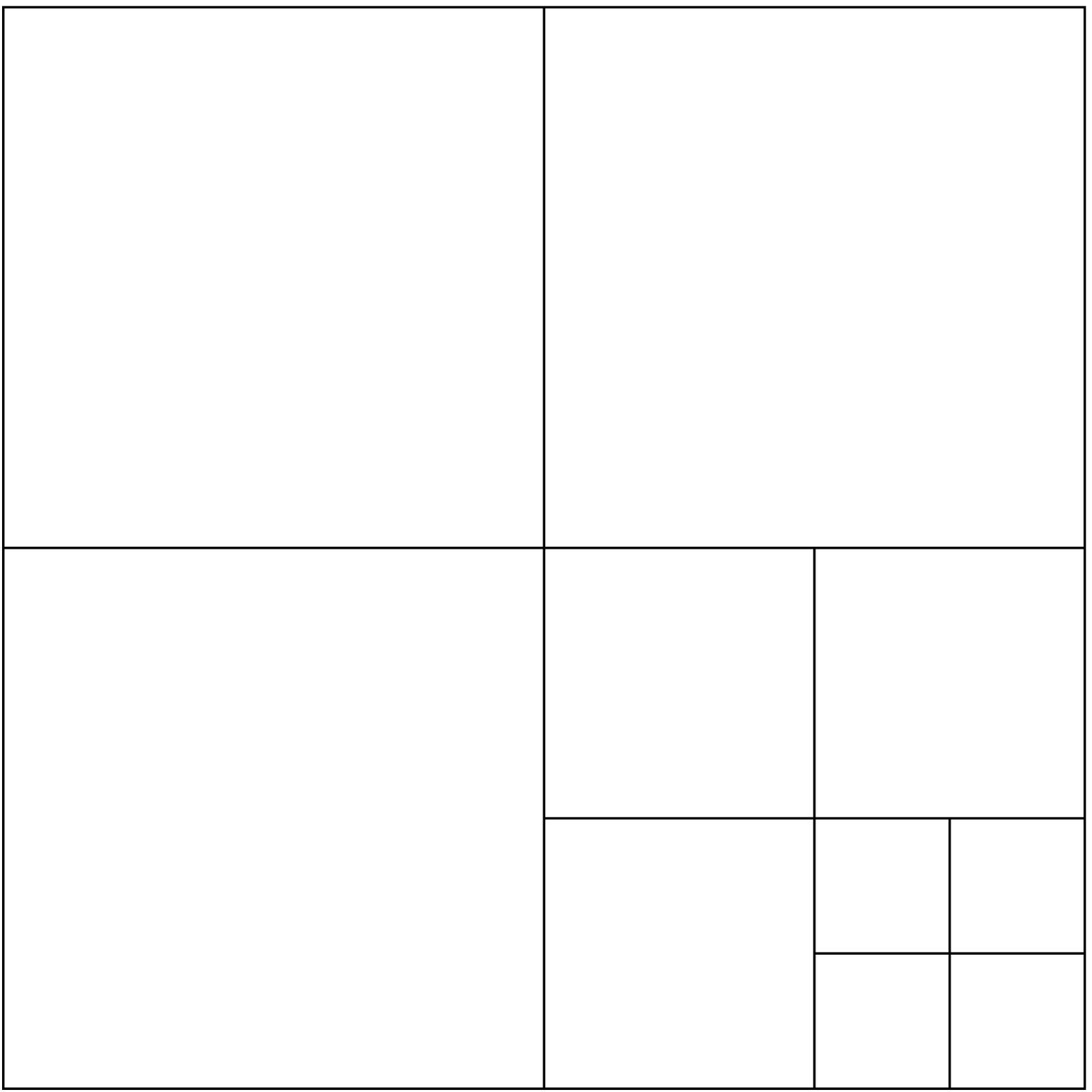}
  \hspace{0.1cm}
  \includegraphics[height=2.2cm]{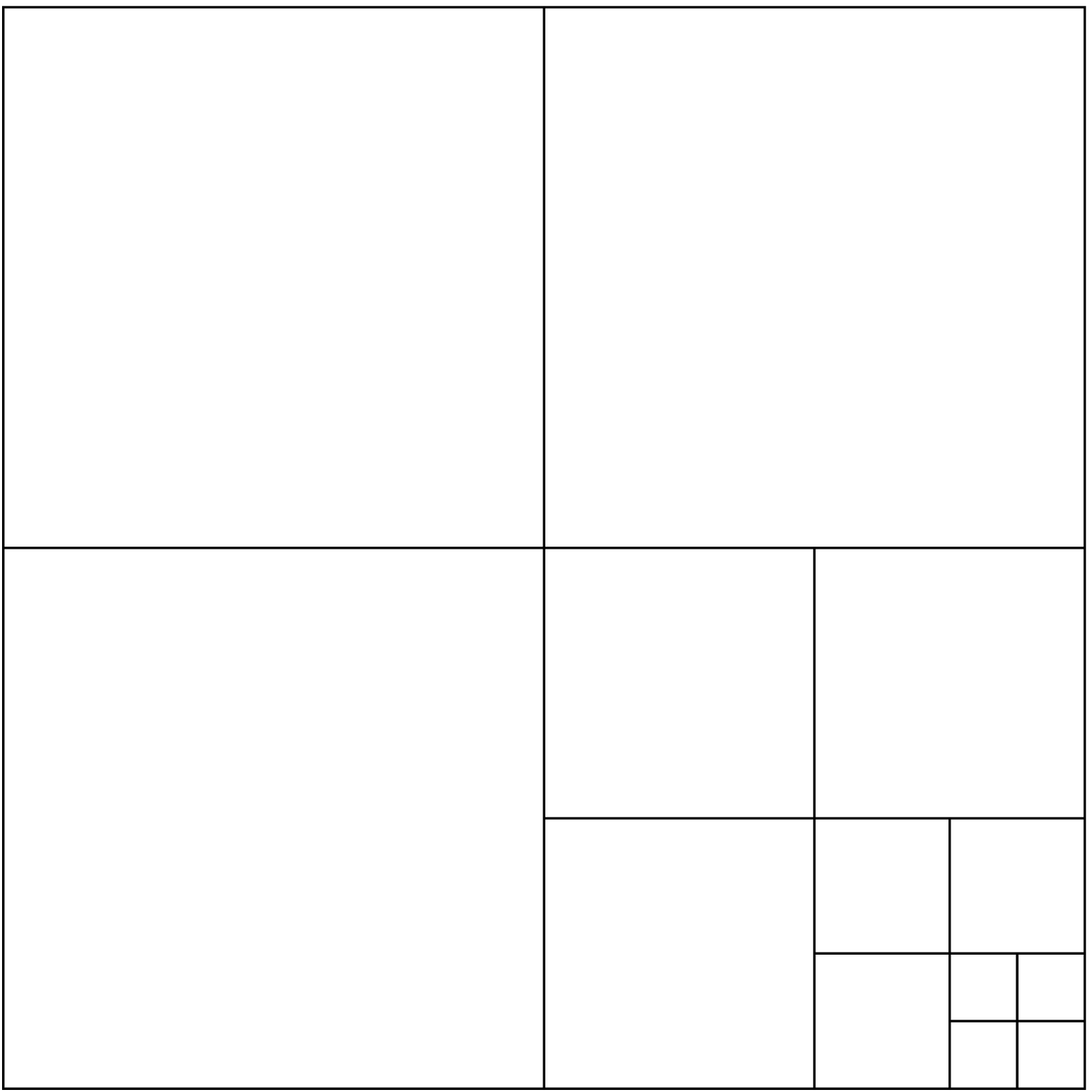}
  \hspace{0.1cm}
  \includegraphics[height=2.2cm]{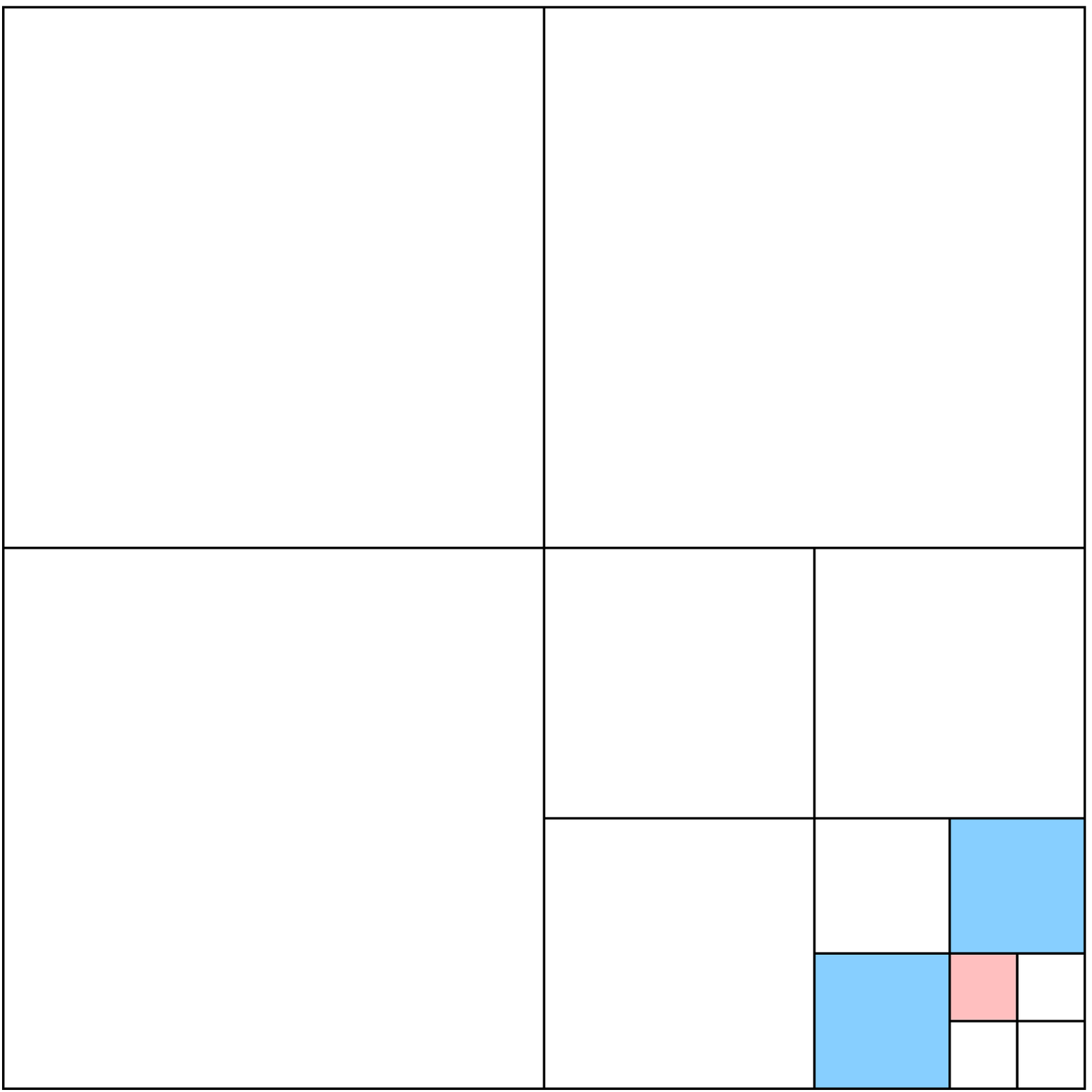}
\end{center}
\caption{The subdivision of the marked (red) box on level $\ell$ would 
  create two boxes (blue) with more than one hanging node at one 
  edge.\label{first_sub}}
\end{figure}

After processing all boxes on level $\ell$, it may occur that
there are boxes on level $\ell-1$ that would have more than 
one hanging node on an edge after refinement of the marked
boxes, cf. Figure \ref{first_sub}. 
Since we want to avoid this, we have to perform a
closure operation for all such elements and for all coarser levels $\ell-1,\ldots,1$.

%\begin{center}
%  \fbox{Part II: Closure and Refinement}
%\end{center}
\begin{list}{Closure}{}
\item[3. {\bf(Close)}]
  Let ${\cal L}^{(\ell-1)}$ be
  the set of all boxes on level $\ell-1$ having too many hanging 
  nodes. All of these are marked for refinement. By construction
  a single refinement of each box is sufficient. 
  However, a refinement on level 
  $\ell-1$ might then produce too many hanging nodes in a box on 
  level $\ell-2$. Therefore, we have to form the lists ${\cal L}^{(j)}$,
  $j=\ell-1,\ldots,1$ of boxes with too many hanging nodes successively 
  on all levels and mark the elements.
\item[4. {\bf(Refine)}]
  At last we refine all boxes (on all levels) that are marked
  for refinement.
\end{list}
The result of the closure
operation is depicted in Figure \ref{closure}.
\begin{figure}[htb]
  \begin{center}
    \includegraphics[height=2.5cm]{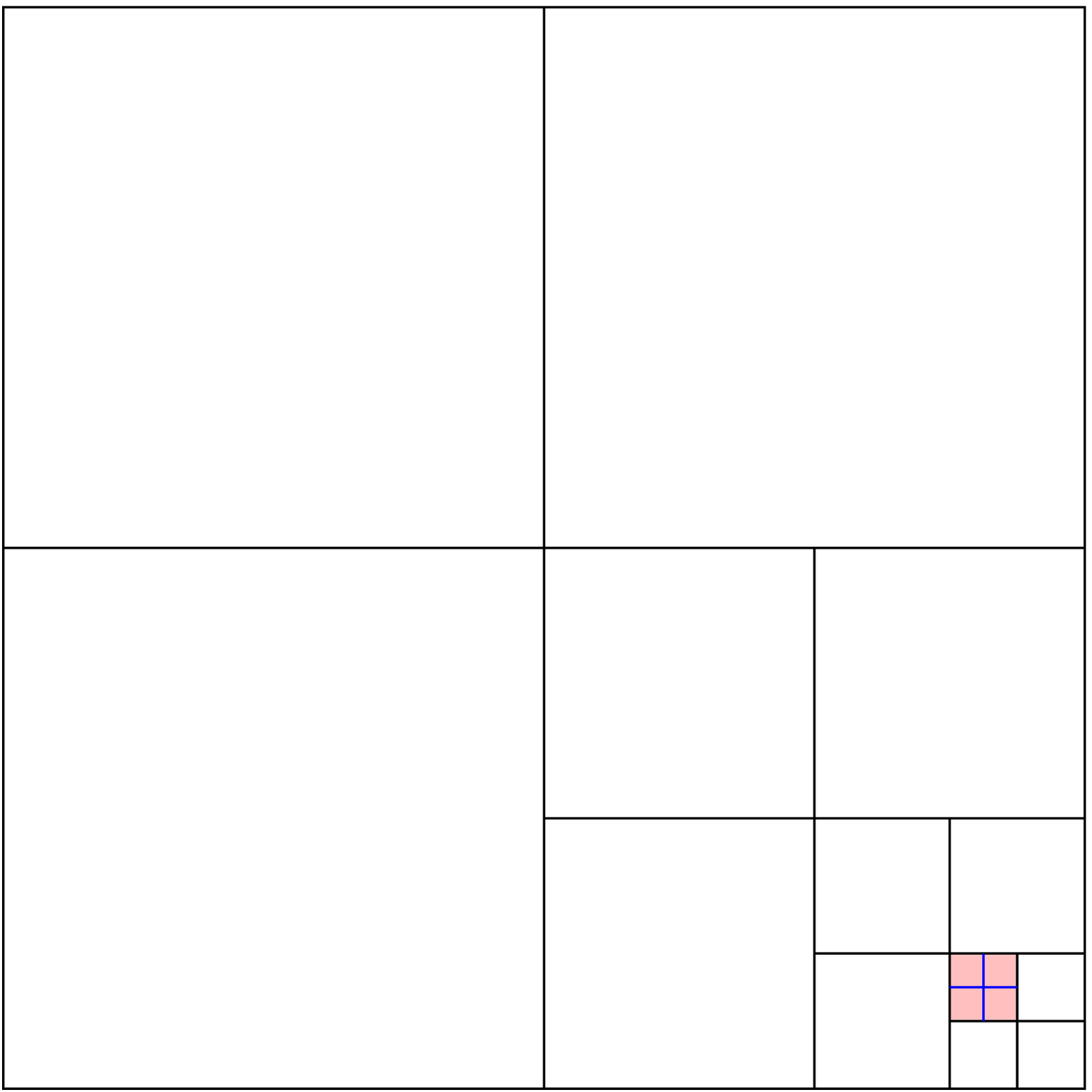}
    \hspace{0.5cm}
    \includegraphics[height=2.5cm]{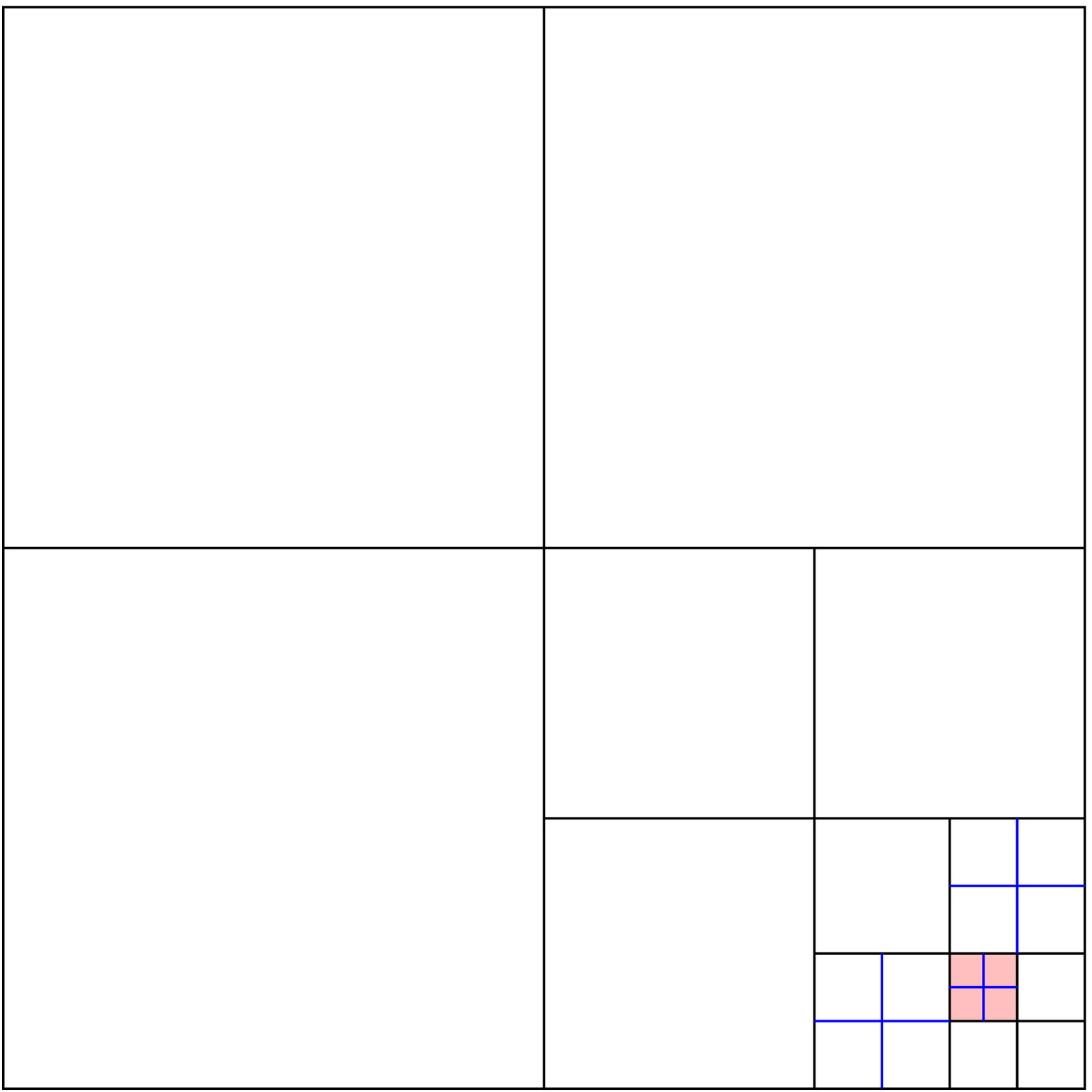}
    \hspace{0.5cm}
    \includegraphics[height=2.5cm]{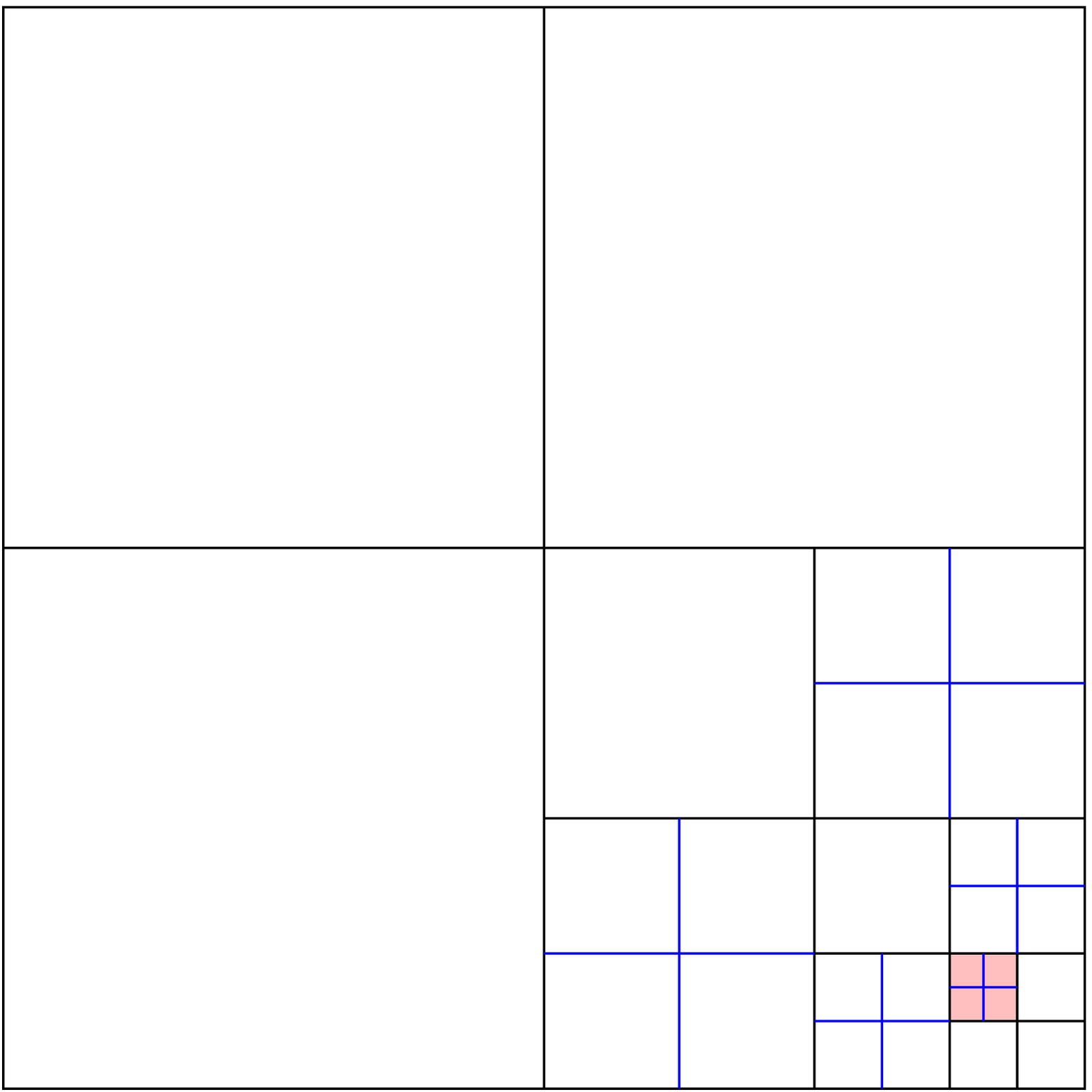}
    \hspace{0.5cm}
    \includegraphics[height=2.5cm]{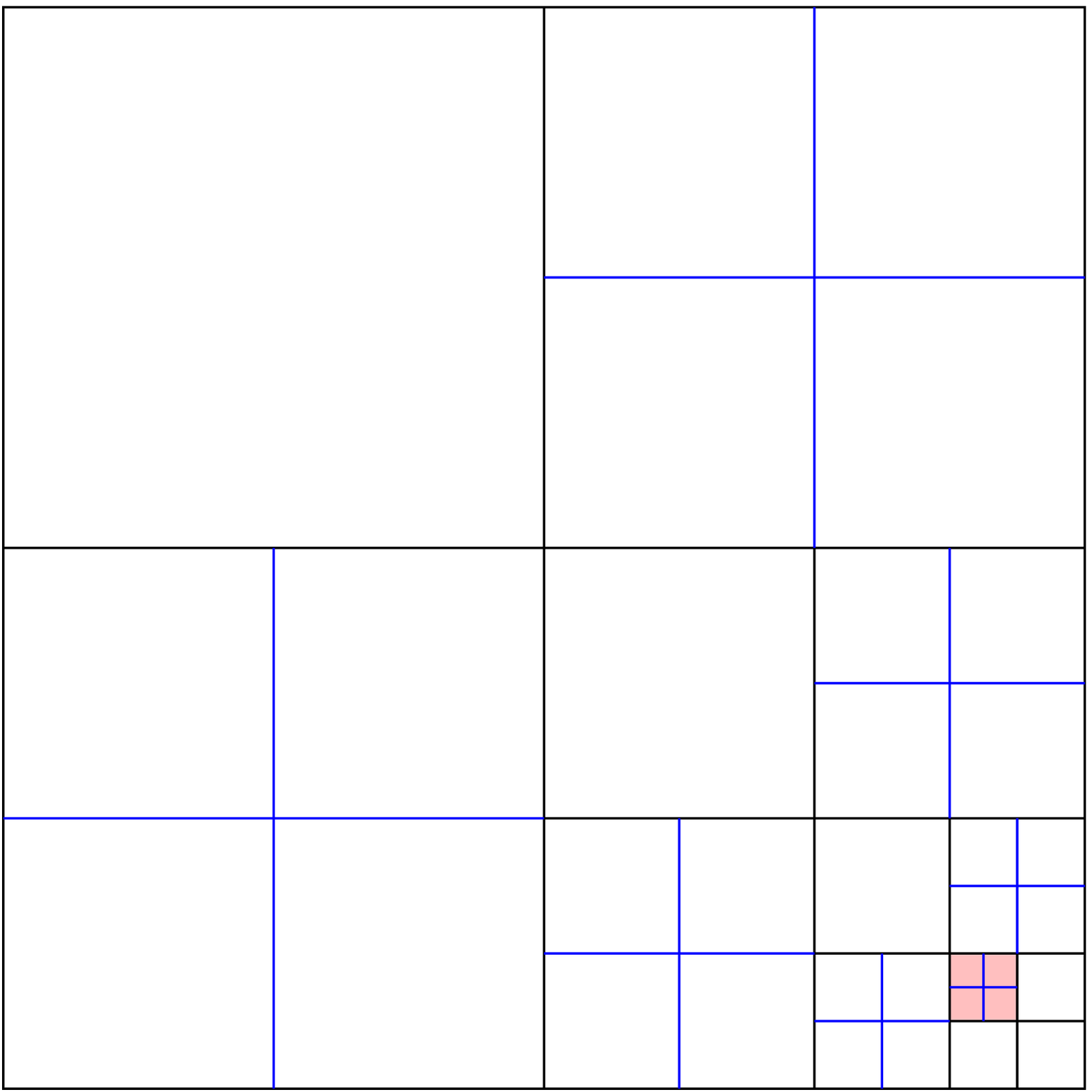}
  \end{center}
  \caption{The subdivision of the red box makes it necessary to 
    subdivide nodes on all levels.\label{closure}}
\end{figure}
Each of the boxes
in the closed grids lives on a unique level $\ell\in\{1,\ldots,J\}$.
It is important that a box is either refined regularly (split into
four successors on the next level) or it is not refined at all. For
each box that is marked in step 1, there are at most ${\cal O}(\log N)$
boxes marked during the closure step 3. 

\begin{lemma} The complexity for the construction and storage of the (finite) box tree with boxes $\mathcal{B}^{\ell}_\nu$ and corresponding 
  cluster tree $\ct$ with clusters  $t_\nu=t(\mathcal{B}^{\ell}_\nu)$ is
  of complexity ${\cal O}(N\log N)$, where $N$ is the number of 
  barycenters, i.e., the number of triangles in the triangulation $\tau$.
\end{lemma}
\begin{proof}
%  On each level of the box tree, by construction, the number of boxes is bounded by the number of barycenters. The depth $p$ of the box tree is due to Assumption \ref{def_p} bounded by $p={\cal O}(\log N)$. Thus the total storage is ${\cal O}(N \log N)$ for the boxes and the corresponding clusters. Since each box is processed at most two times (once from step 1 and once from step 3), the complexity for the construction is also bounded by the storage complexity.
  For the level $\ell$ of the tree,  let $n_{\ell}$ be the number of leaf boxes
  and $m_{\ell}$ be the boxes which have child boxes. Accordingly, the total number of
  the boxes on level $\ell$ is $n_{\ell}+m_{\ell}$.  By definition,
$$
n_{\ell}+m_{\ell} = 4m_{\ell-1},
$$
%is subdivided by $m_{\ell}-1$ nodes on level $\ell-1$, i.e.
where $\ell\geq 2$ and $n_{1}+m_{1}=m_{1}=1$.
Since $\sum_{\ell=1}^{J}n_{\ell}\lesssim N,$ we have 
$$
N\gtrsim\sum_{\ell=1}^{J}n_{\ell} = \sum_{\ell=2}^{J}(4m_{\ell-1} -m_{\ell})
=3\sum_{\ell=1}^{J}m_{\ell}+1.
$$
As a result,
$$
\sum_{\ell=1}^{J}m_{\ell}\lesssim N.
$$
The total work for generating the tree is $\sum_{\ell=1}^{J}n_{\ell}+m_{\ell}\lesssim N$.

Given $\ell$, let $\alpha_{\ell}$ denote the number of boxes in
$\mathcal{L}^{(\ell-1)}$ (the set of boxes that have more than 1 hanging
node).  Since every box in $\mathcal{L}^{(\ell-1)}$ has to be a leaf
box, we have $\alpha_{\ell}\leq n_{\ell}$.  As the process of closing
each hanging node will go through at most two boxes in any given
level, the total number of the marked boxes in this closure process is
bounded by
$$
\sum_{\ell=1}^{J}2J\alpha_{\ell}\lesssim JN\lesssim N\log N.
$$
\end{proof}

\subsection{Construction of a conforming auxiliary grid hierarchy}

At last, we create a hierarchy of nested conforming triangulations
by subdivision of the boxes and by discarding those boxes that lie
outside the domain $\Omega$. 
For any box $\mathcal{B}_\nu$ there
can be at most one hanging node per edge. The possible situations
and corresponding local closure operations are presented in 
Figure \ref{closure2}.
\begin{figure}[htb]
\begin{center}
  \includegraphics[height=3cm]{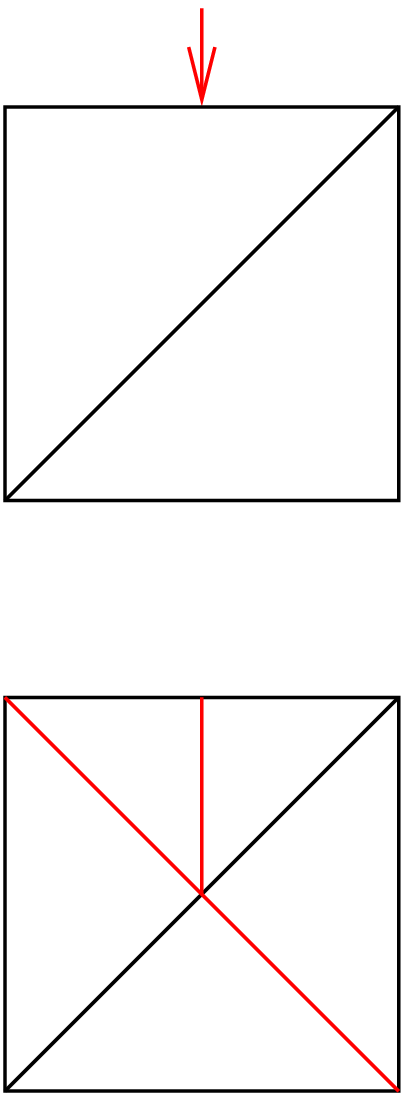}
  \hspace{1.0cm}
  \includegraphics[height=3cm]{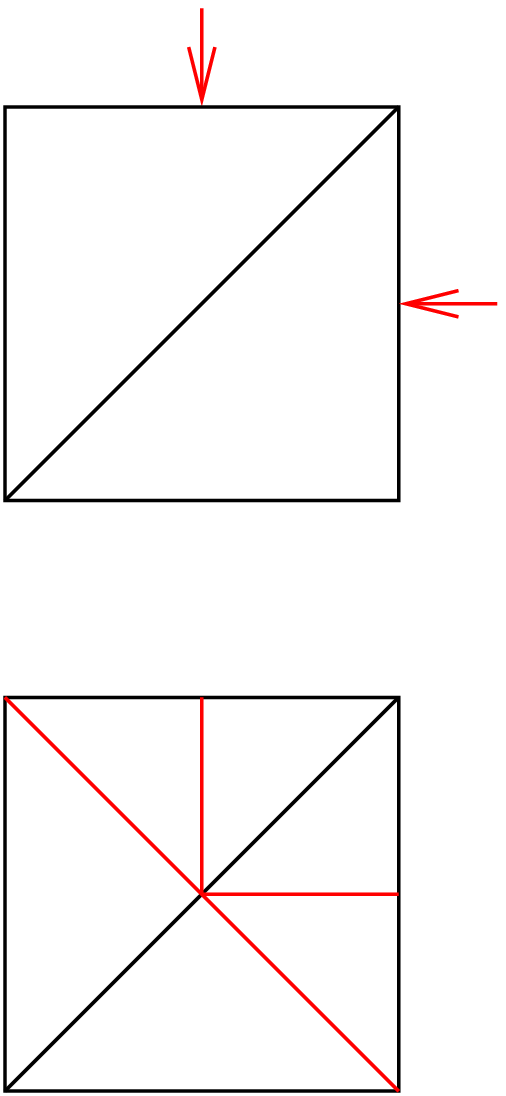}
  \hspace{0.6cm}
  \includegraphics[height=3cm]{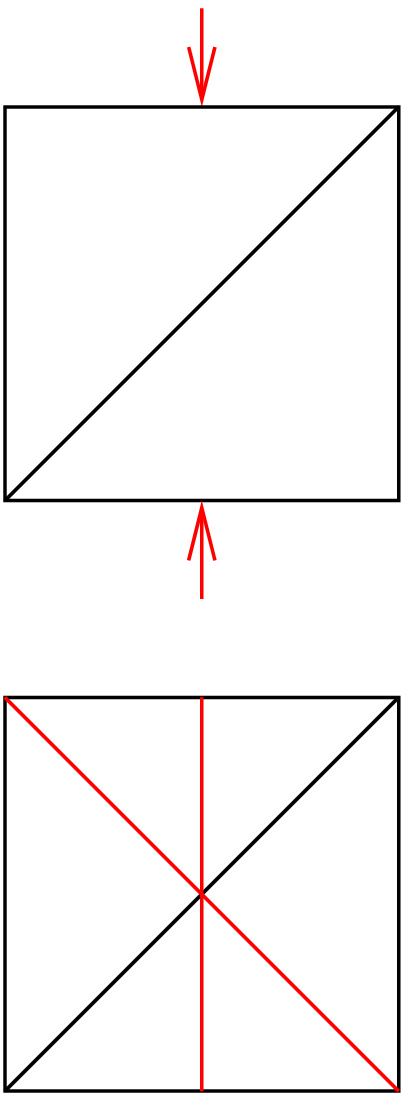}
  \hspace{1.0cm}
  \includegraphics[height=3cm]{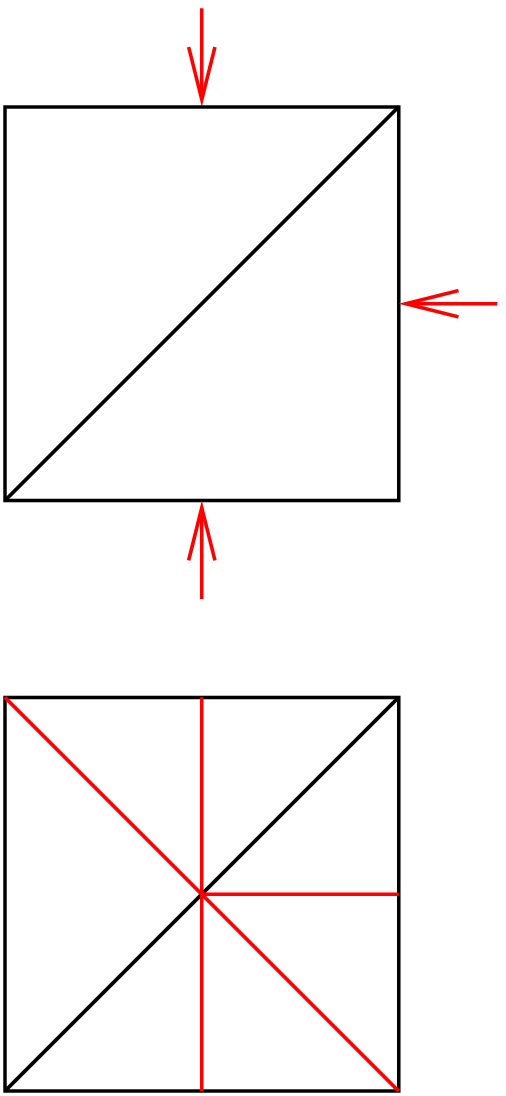}
  \hspace{0.2cm}
  \includegraphics[height=3cm]{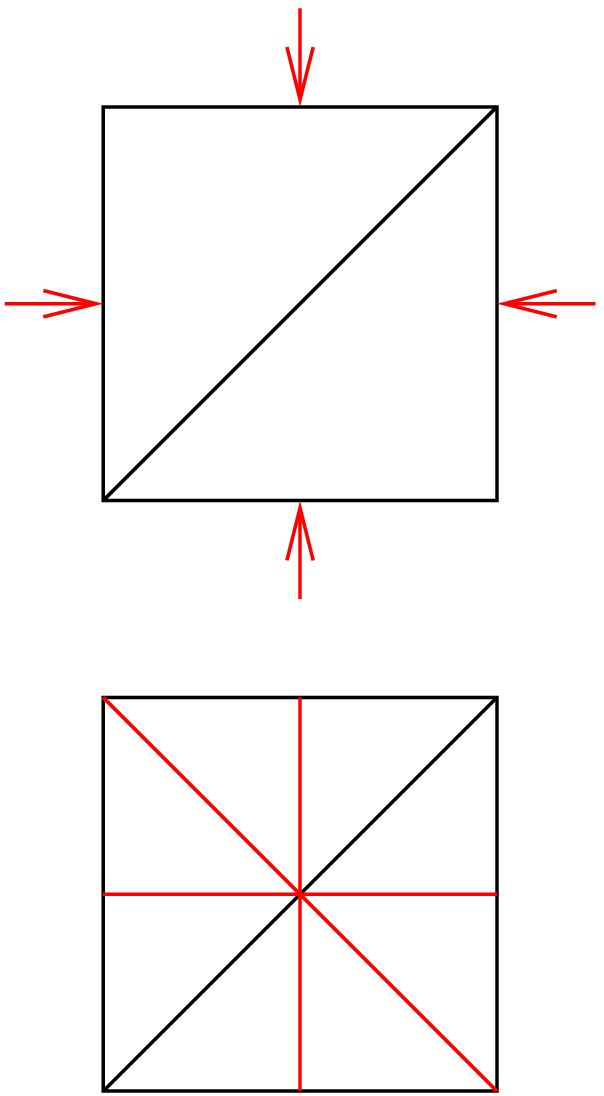}
\end{center}
\caption{Hanging nodes can be treated by a local subdivision 
  within the box $\mathcal{B}_\nu$. The top row shows a box 
  with $1, 2, 2, 3, 4$ hanging nodes, respectively, and the 
  bottom row shows the corresponding triangulation of the box.\label{closure2}}
\end{figure}
The closure operation introduces new elements on the next finer
level.

The final hierarchy of triangular grids $\sigma^{(1)},\ldots,\sigma^{(J)}$ is
nested and conforming without hanging nodes. All triangles have
a minimum angle of $45$ degrees, i.e., they are shape-regular,
cf. Figure \ref{refine}. 

\begin{figure}[htb]
\begin{center}
  \includegraphics[height=2.cm]{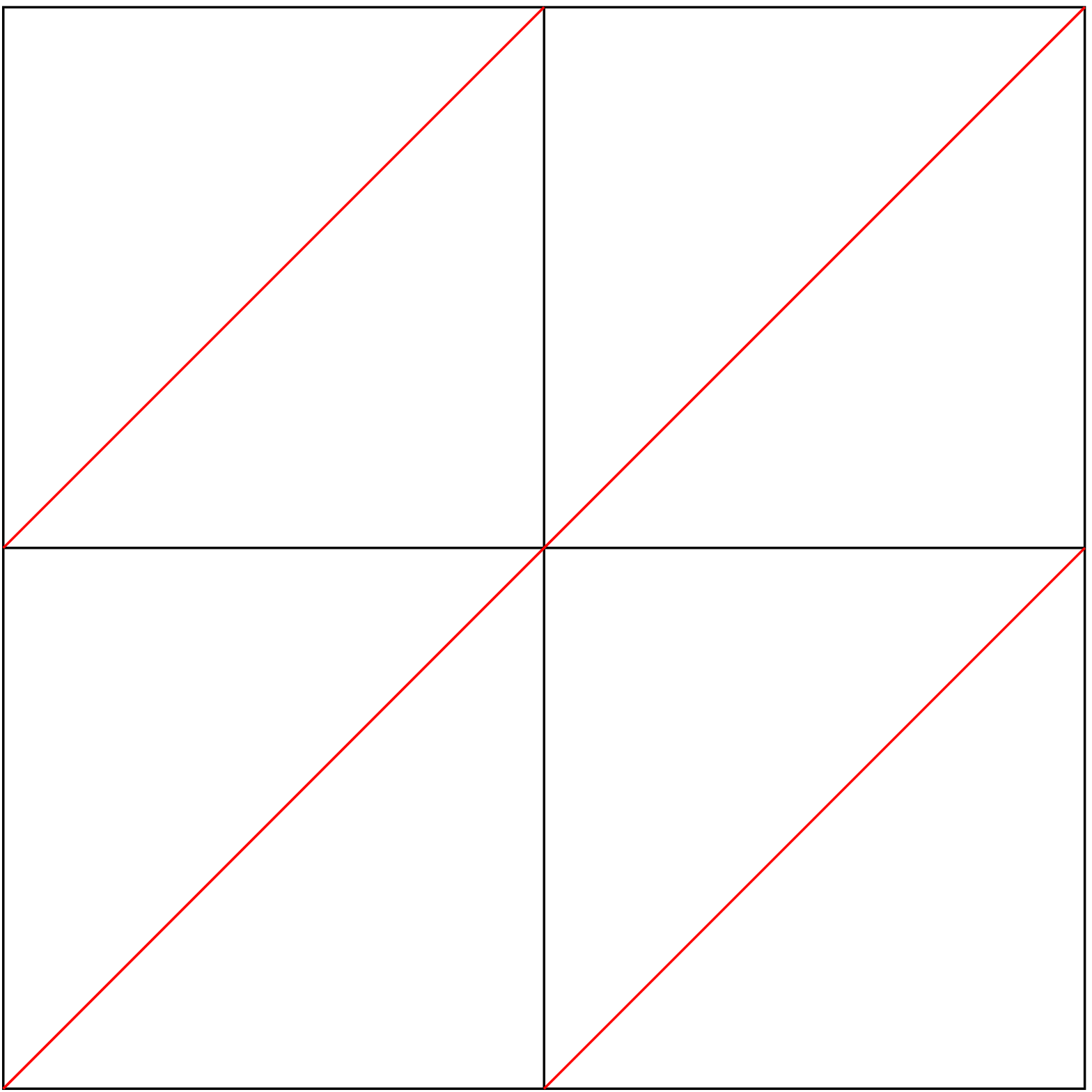}
  \hspace{0.1cm}
  \includegraphics[height=2.cm]{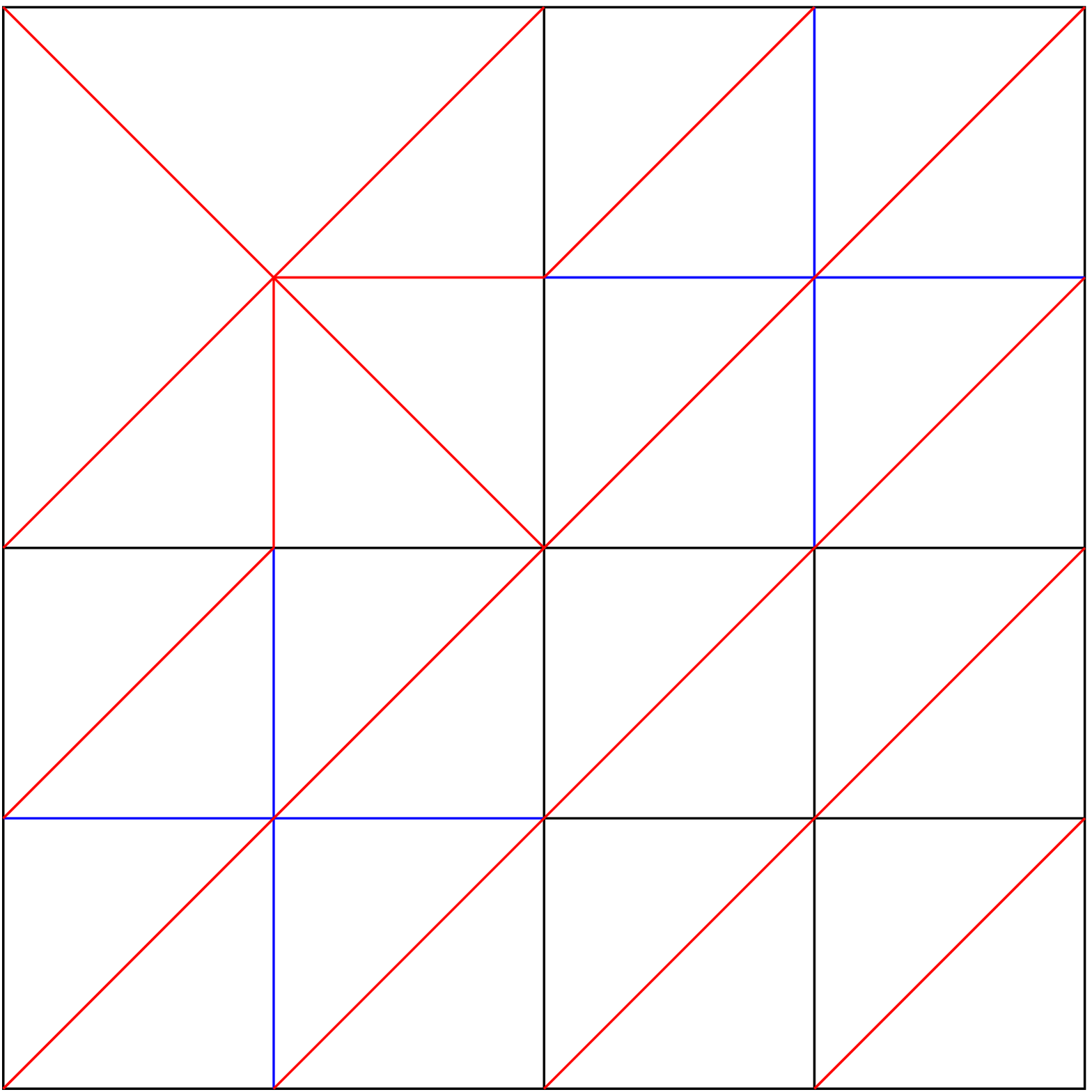}
  \hspace{0.1cm}
  \includegraphics[height=2.cm]{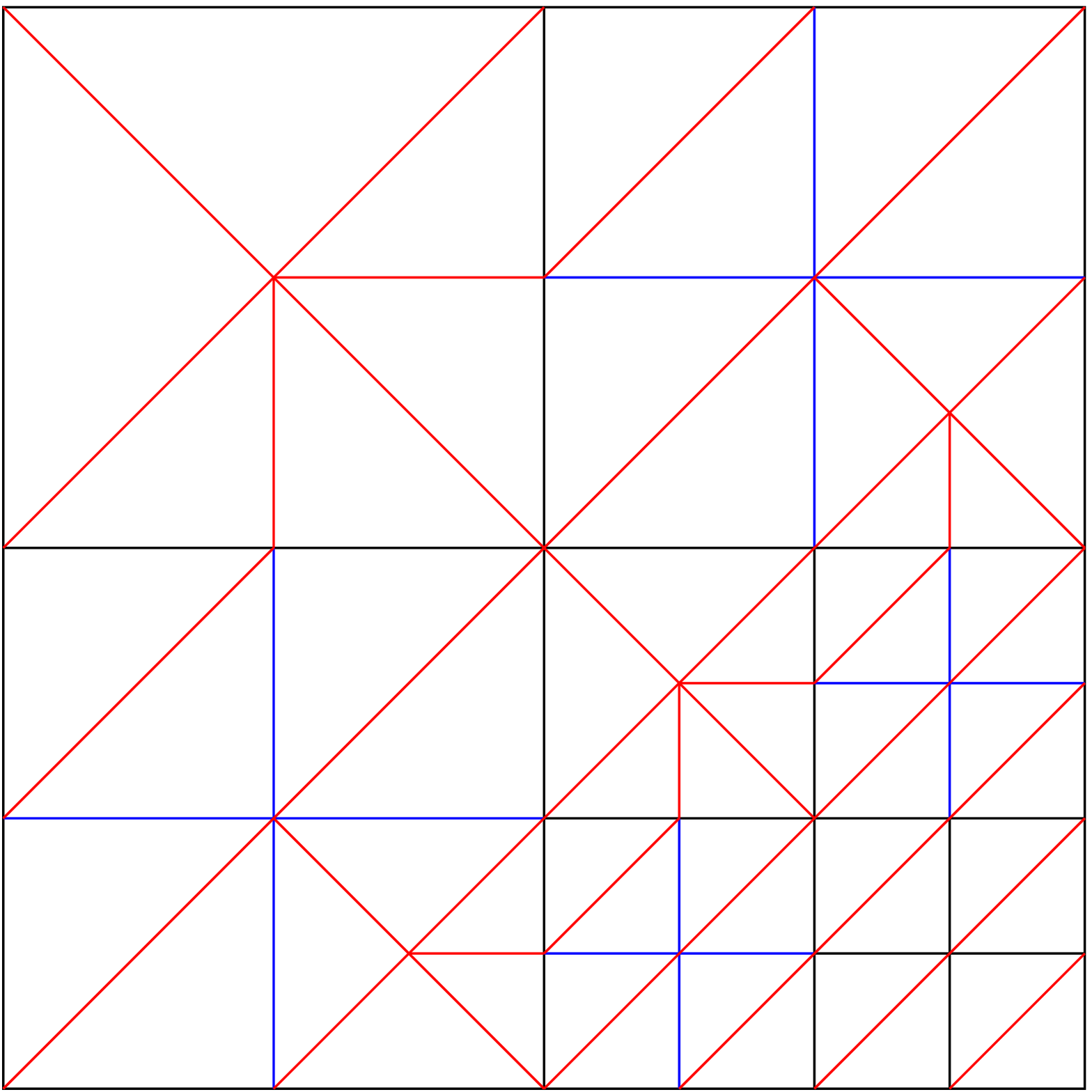}
  \hspace{0.1cm}
  \includegraphics[height=2.cm]{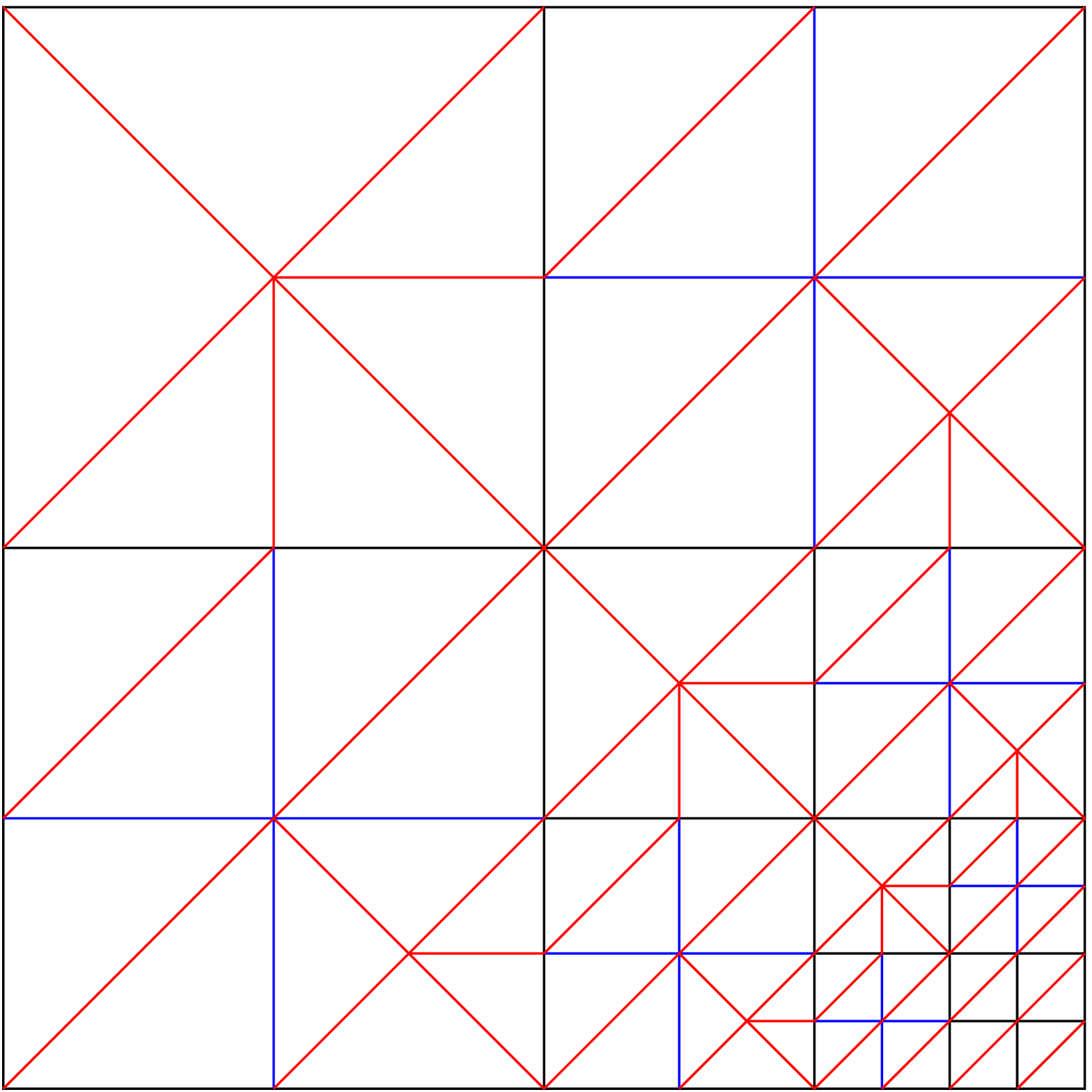}
  \hspace{0.1cm}
  \includegraphics[height=2cm]{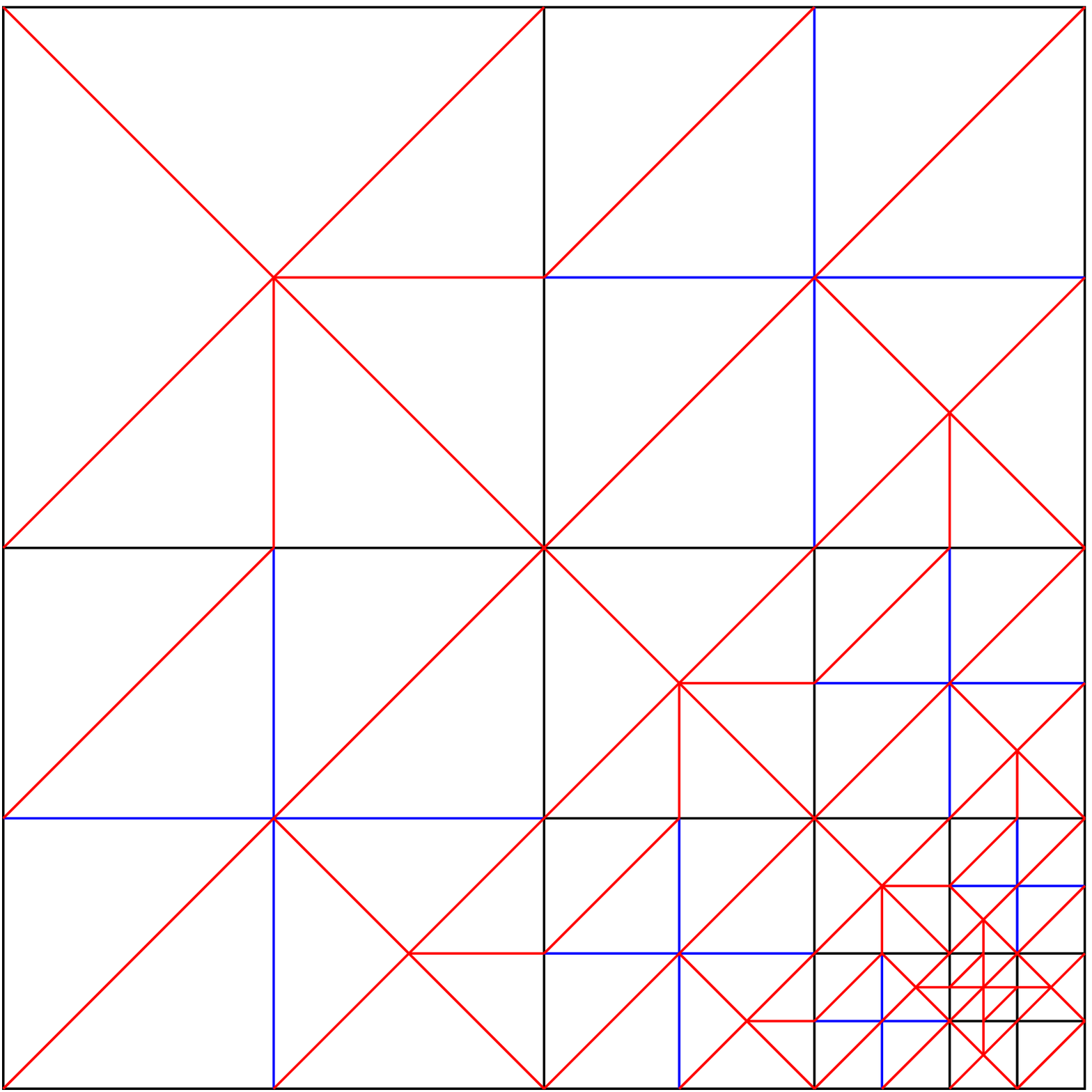}\\
  level 1 \hspace{1.2cm}
  level 2 \hspace{1.2cm}
  level 3 \hspace{1.2cm}
  level 4 \hspace{1.2cm}
  level 5 
\end{center}
\caption{The final hierarchy of nested grids. Red edges were
  introduced in the last (local) closure step.\label{refine}}
\end{figure}

The triangles in the quasi-regular meshes $\sigma^{(1)},\ldots,\sigma^{(J)}$
have the following properties: 
\begin{enumerate}
\item All triangles in $\sigma^{(1)},\ldots,\sigma^{(J)}$ that have 
  children which are themselves further subdivided, are refined 
  regularly (four congruent successors) as depicted here.
  \begin{center}
    \includegraphics[width=9cm]{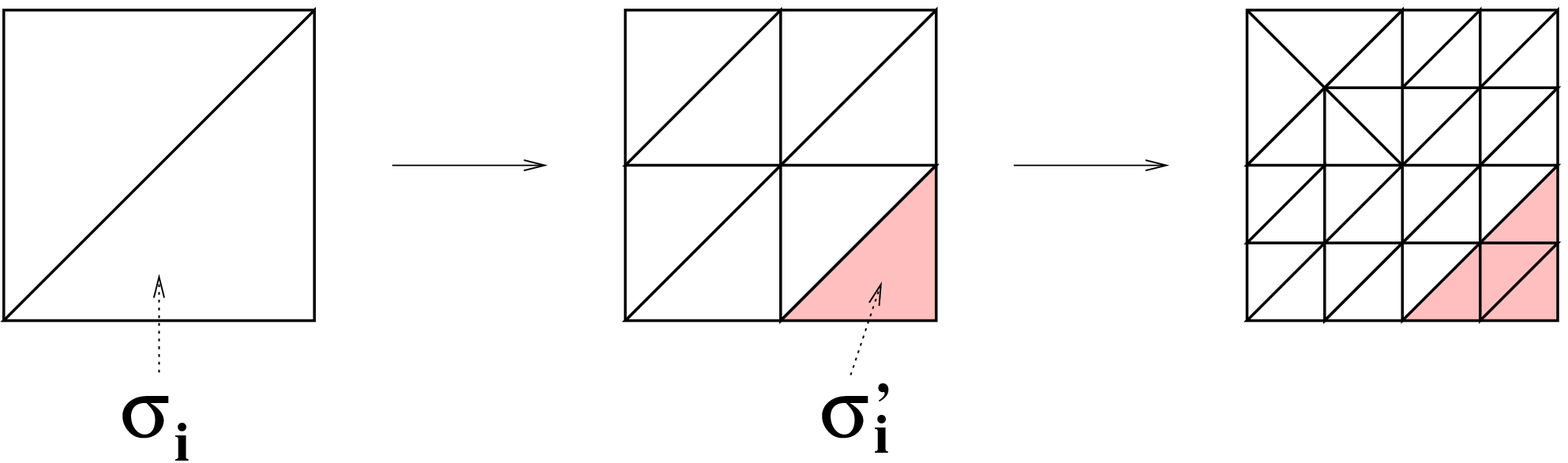}
  \end{center}
\item Each triangle $\sigma_i\in\sigma^{(j)}$ that is subdivided 
  but not regularly refined, has successors $\sigma_i'$ that will 
  not be further subdivided.
%%  \begin{center}
%%    \includegraphics[width=9cm]{trisub2.eps}
%%  \end{center}
\end{enumerate}

The hierarchy of grids constructed so far covers on each level the
whole box $\cal B$. This hierarchy has now to be adapted to the
boundary of the given domain $\Omega$.
%The treatment of Dirichlet and Neumann boundary conditions is different and will be explained in the following two subsections.
In order to explain the construction we will consider the domain
$\Omega$ and triangulation $\mathcal{T}$ ($5837$ triangles) from
Figure $\ref{Fig:baltic}$.
\begin{figure}[!htb]
\begin{center}
  \includegraphics[height=8cm]{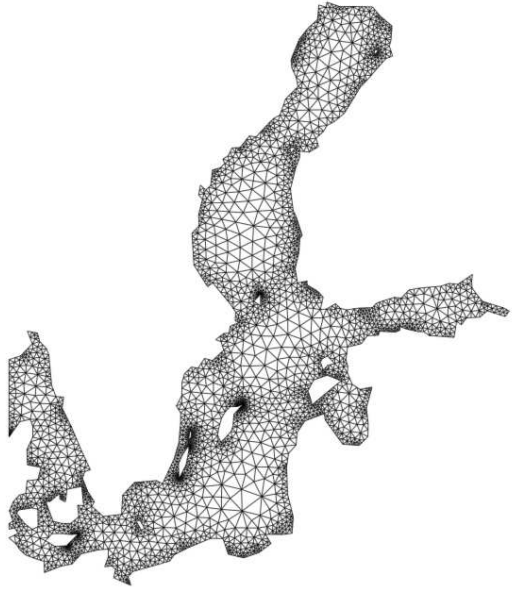}
\end{center}
\caption{A triangulation of the Baltic sea with local refinement and small
  inclusions.\label{Fig:baltic}}
\end{figure}
The triangulation consists of shape-regular elements, it is locally refined,
it contains many small inclusions, and the boundary $\Gamma$ of the domain $\Omega$ is rather complicated. 

\subsection{Adaptation of the auxiliary grids to the boundary}

\textbf{The Dirichlet boundary: } On the Dirichlet boundary we want to satisfy homogeneous boundary
conditions (b.c.), i.e., $u|_\Gamma=0$ (non-homogeneous b.c. can
trivially be transformed to homogeneous ones). On the given fine 
triangulation $\tau$ this is achieved by use of basis functions 
that fulfil the b.c. Since the auxiliary triangulations 
$\sigma^{(1)},\ldots,\sigma^{(J)}$ do not necessarily resolve the
boundary, we have to use a slight modification.

\begin{definition}[Dirichlet auxiliary grids]
  We define the auxiliary triangulations $\mathcal{T}^{D}_{\ell}$ by
  \[\mathcal{T}^{D}_{\ell} := \{\tau \in \sigma^{(\ell)}\mid \tau\subset\Omega\},
  \quad
  \ell=1,\ldots,J.
  \]
\end{definition}

\begin{figure}[htb]
  \begin{center}
    \includegraphics[height=4.2cm]{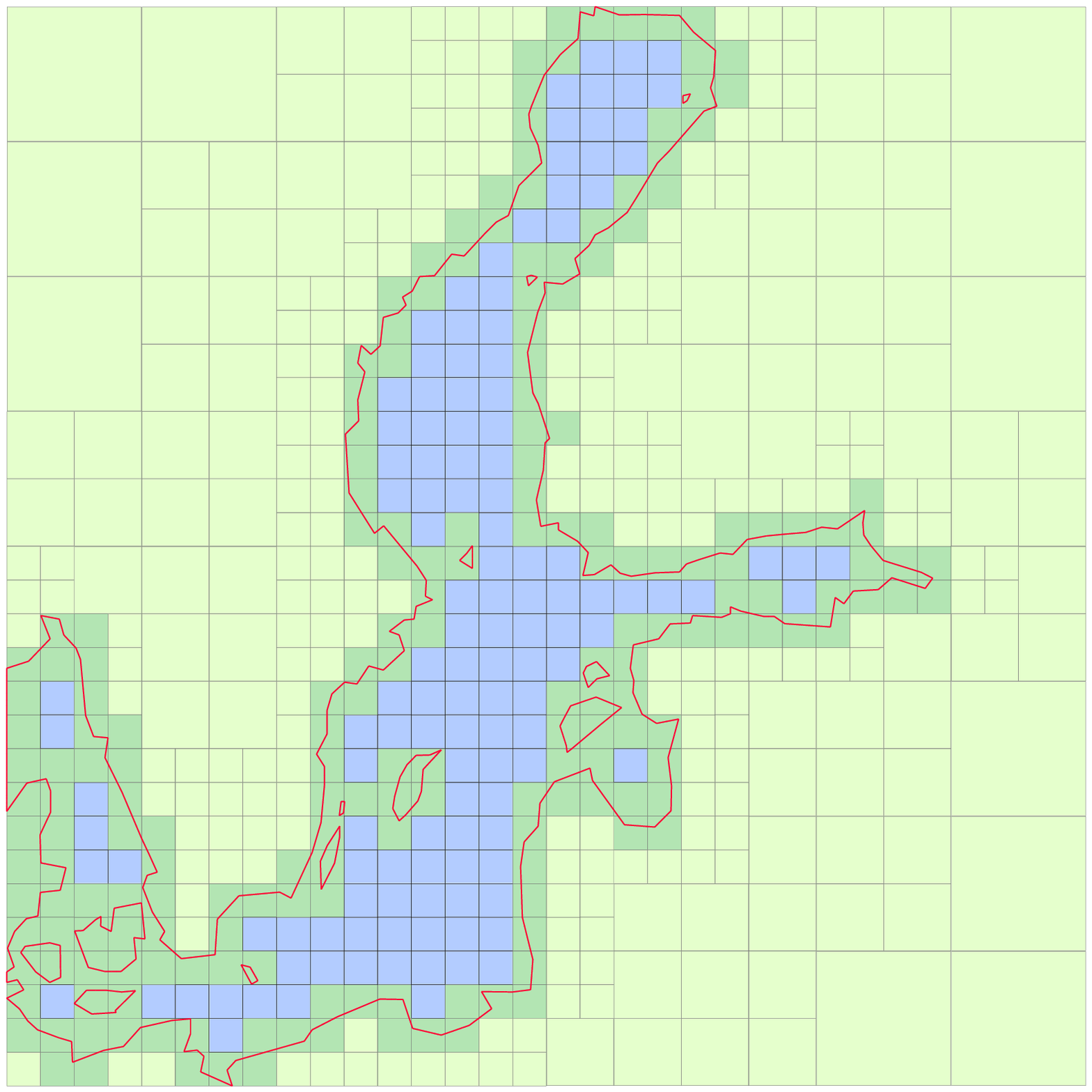}
    \hspace{-0.5cm}
    \includegraphics[height=4.2cm]{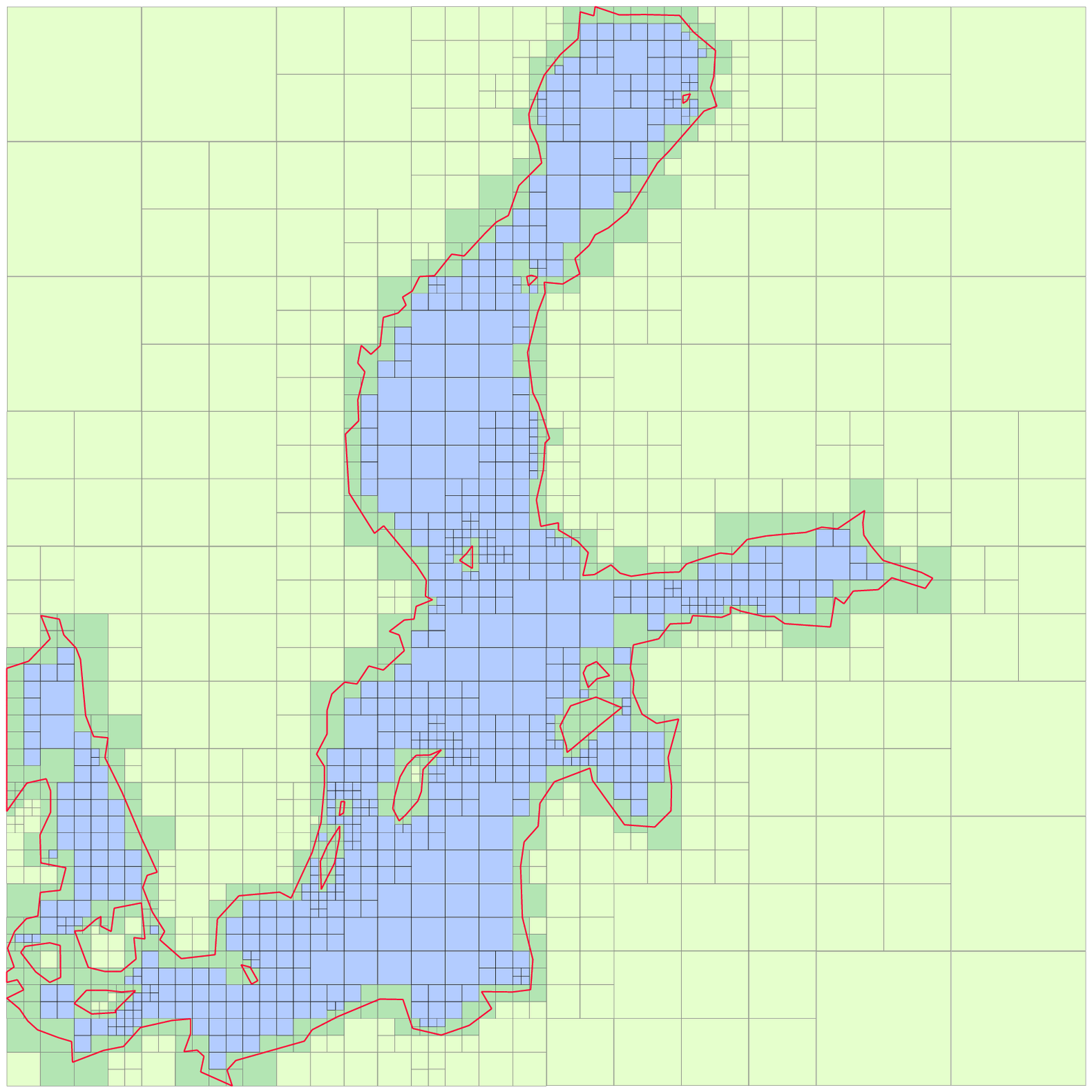}
    \hspace{-0.5cm}
    \includegraphics[height=4.2cm]{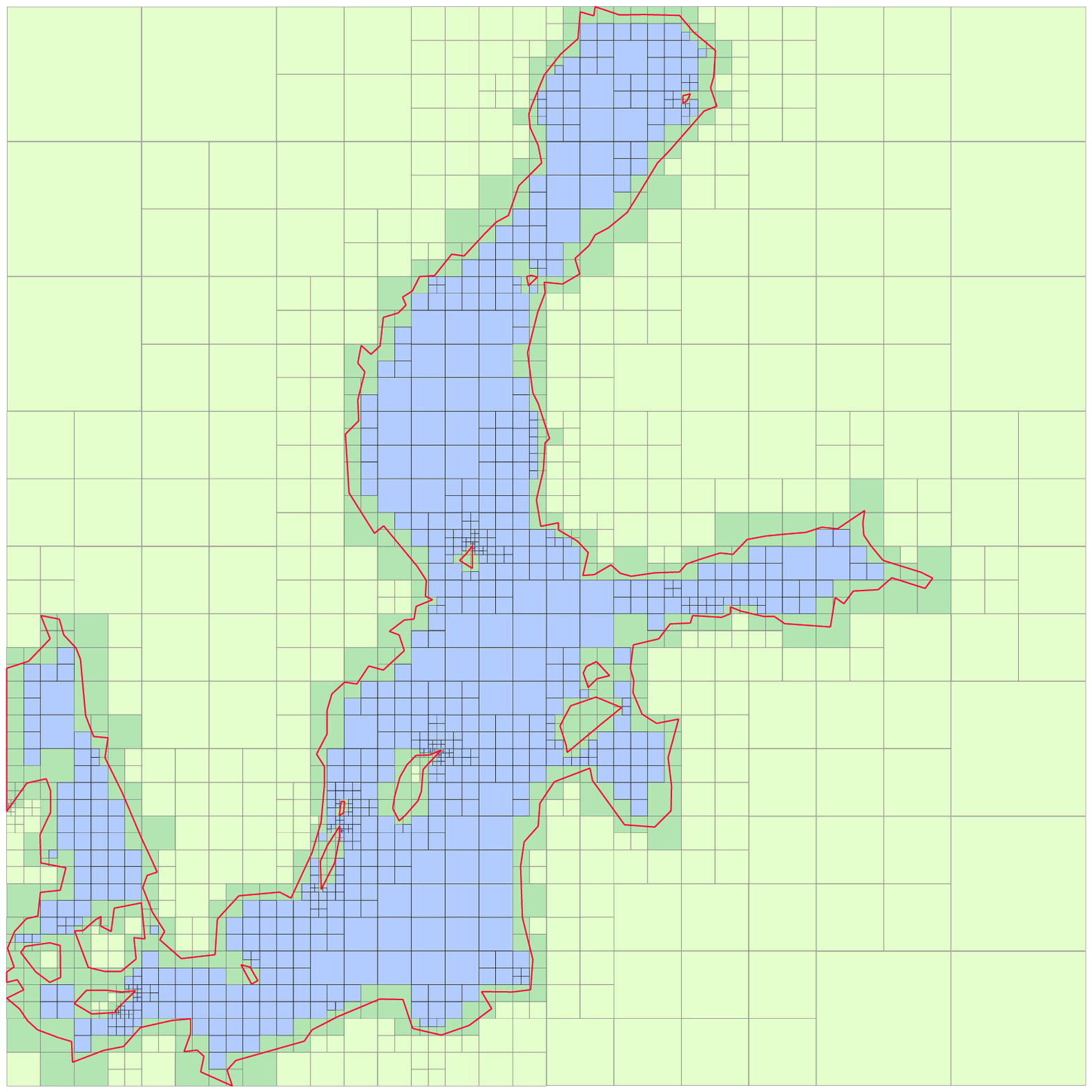}\\
    level 3 \hspace{2.9cm}
    level 4 \hspace{2.9cm}
    level 5 \hspace{2.9cm}
  \end{center}
  \caption{The boundary $\Gamma$ of $\Omega$ is drawn as a red line, boxes
    non-intersecting $\Omega$ are light green, boxes intersecting $\Gamma$ are
    dark green, and all other boxes (inside of $\Omega$) are blue.\label{Fig:boundary}}  
\end{figure}
  
In Figure \ref{Fig:boundary} the Dirichlet auxiliary grids are formed by
the blue boxes. All other elements (light green and dark green) are not
used for the Dirichlet problem. On an auxiliary grid we impose homogeneous
Dirichlet b.c. on the boundary 
\[\Gamma_{\ell} := \partial \Omega^{D}_{\ell} ,
\qquad
\Omega^{D}_{\ell} := \cup_{\tau\in\mathcal{T}^{D}_{\ell}} \bar\tau.
\]
The auxiliary grids are still nested, but the area covered by the 
triangles grows with increasing the level number:
\[
\Omega^{D}_1\subset \cdots \subset\Omega^{D}_J\subset\Omega,
\qquad
\Omega^{D}_{\ell}:= \bigcup\{\tau\in\mathcal{T}^{D}_{\ell}\}
\]

\textbf{The Neumann boundary: } On the Neumann boundary we want to satisfy natural (Neumann) b.c., i.e., 
$\partial_nu|_\Gamma=0$. 
%On the given fine triangulation, $\tau$ is achieved by use of the standard basis functions, and 
For the auxiliary triangulations $\sigma^{(1)},\ldots,\sigma^{(J)}$, we will approximate
the true b.c. by the natural b.c. on an auxiliary boundary.

\begin{definition}[Neumann auxiliary grids]
Define the auxiliary triangulations $\mathcal{T}^{N}_{1},\ldots,\mathcal{T}^{N}_{J}$ by
  \[\mathcal{T}^{N}_{\ell} := \{\tau \in \sigma^{(\ell)}\mid \tau\cap\Omega\ne\emptyset\},
  \quad
  \ell=1,\ldots,J.
  \]
\end{definition}

\begin{figure}[htb]
  \begin{center}
    \includegraphics[height=4.2cm]{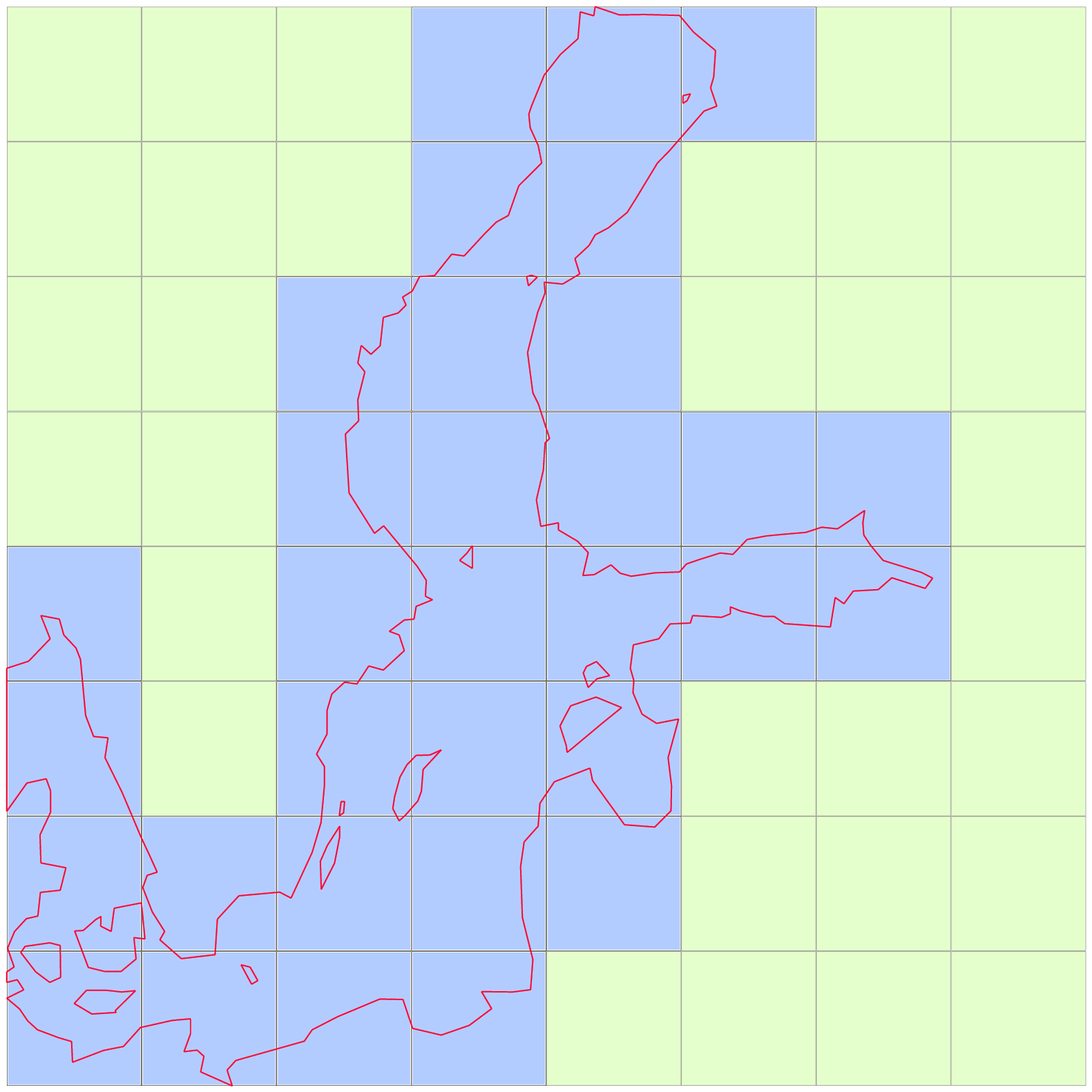}
    \hspace{-0.5cm}
    \includegraphics[height=4.2cm]{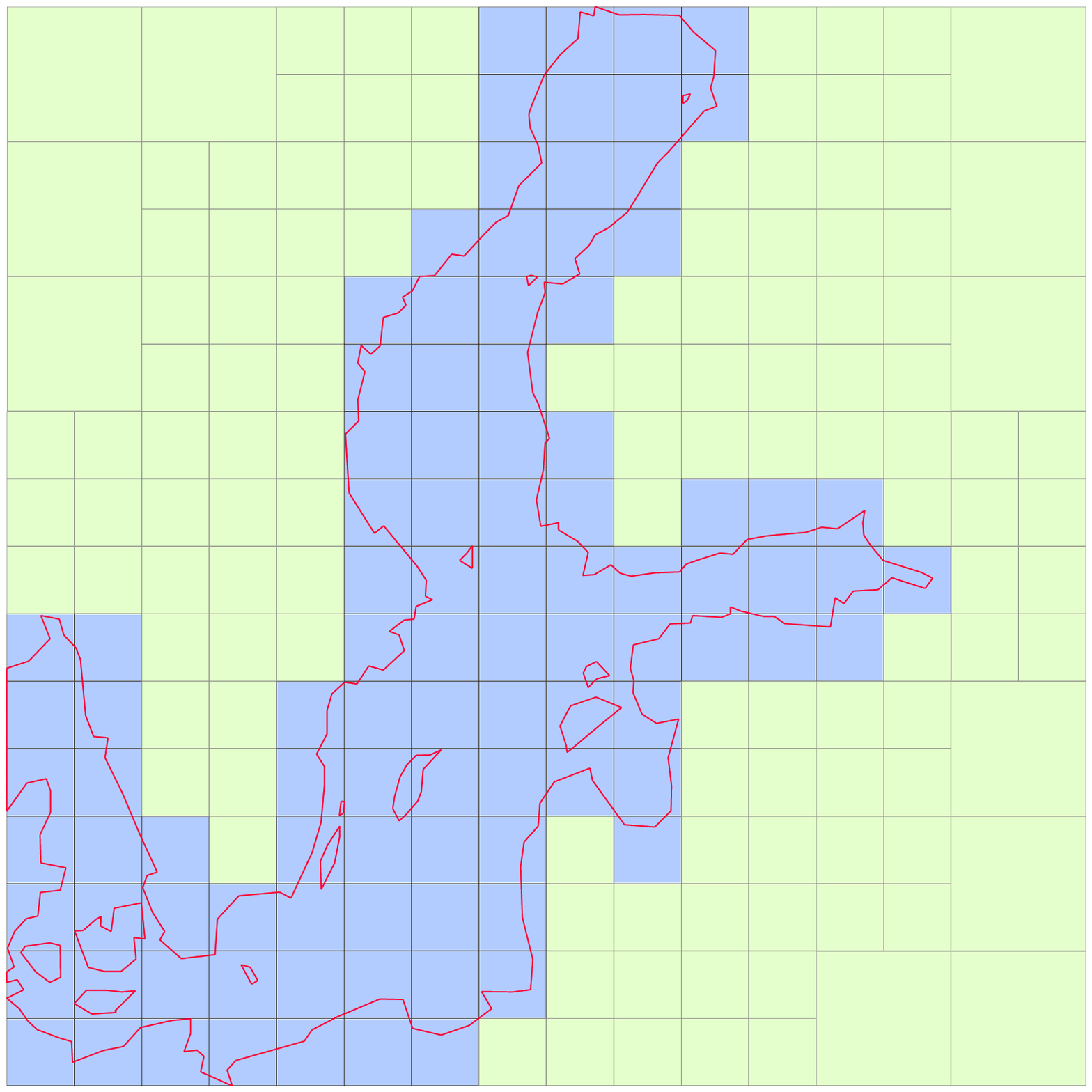}
    \hspace{-0.5cm}
    \includegraphics[height=4.2cm]{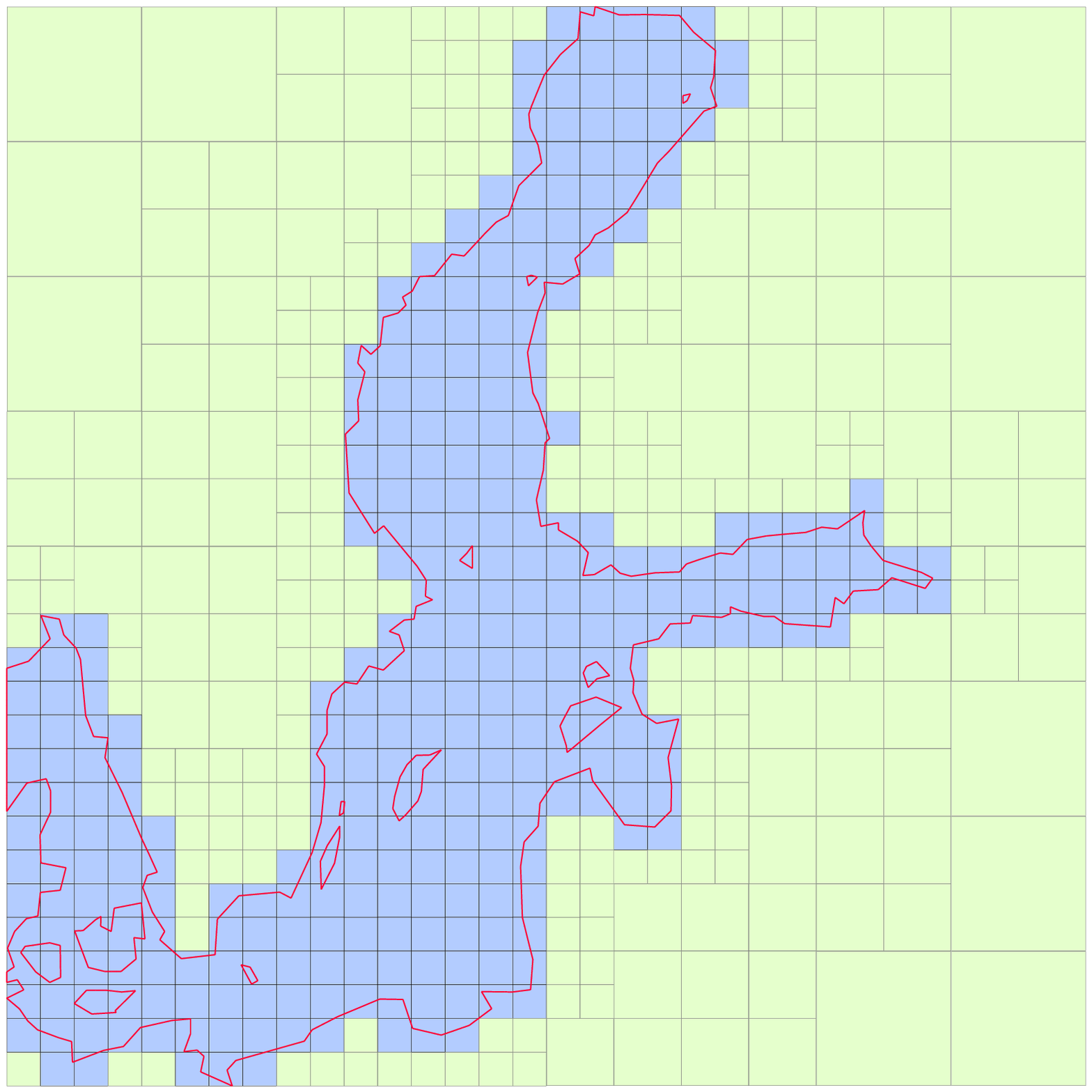}\\
    level 3 \hspace{2.9cm}
    level 4 \hspace{2.9cm}
    level 5 \hspace{2.9cm}
  \end{center}
  \caption{The boundary $\Gamma$ of $\Omega$ is drawn as a red line, boxes
    non-intersecting $\Omega$ are light green, and all other boxes 
    (intersecting $\Omega$) are blue.\label{Fig:Nboundary}}  
\end{figure}
  
In Figure \ref{Fig:Nboundary} the Neumann auxiliary grids are formed by
the blue boxes. All other elements (light green) are not
used for the Neumann problem. On an auxiliary grid we impose 
natural Neumann b.c., the auxiliary grids are non-nested. 
The area covered by the triangles grows with decreasing level number:
\[
\Omega \subset \Omega^{N}_{J}\subset \cdots \subset\Omega^{N}_{1},
\qquad
\Omega^{N}_{\ell}:= \bigcup\{\tau\in\mathcal{T}^{N}_{\ell}\}
\]
\begin{remark}[Mixed Dirichlet/Neumann b.c.]
  The defintion of the grids for mixed boundary conditions of Dirichlet
  (on $\Gamma_D$) and Neumann type we use the grids
  \[\mathcal{T}^{M}_{\ell} := \{\tau \in \sigma^{(\ell)}\mid \tau\cap\Omega\ne\emptyset
  \mbox{ and } \tau\cap\Gamma_D=\emptyset\},
  \quad
  \ell=1,\ldots,J.
  \]
  The b.c. on the auxiliary grid are of Neumann type 
  except for neighbours of boxes $\sigma\cap\Gamma_D\ne\emptyset$
  where essential Dirichlet b.c. are imposed.
\end{remark}

\begin{figure}[!htb]
  \begin{center}
    \includegraphics[height=7cm]{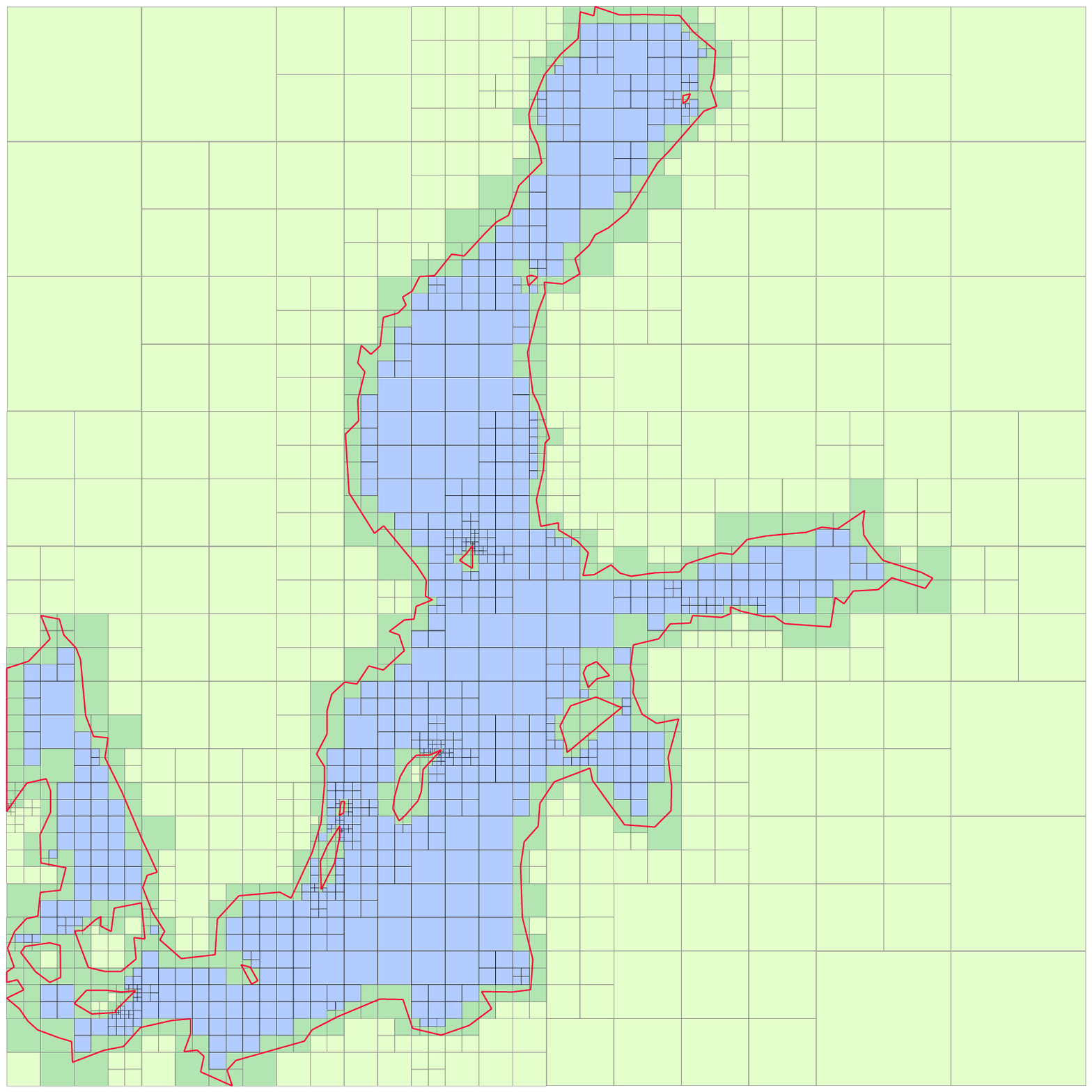}\nolinebreak
    \hspace{-1cm}
    \includegraphics[height=7cm]{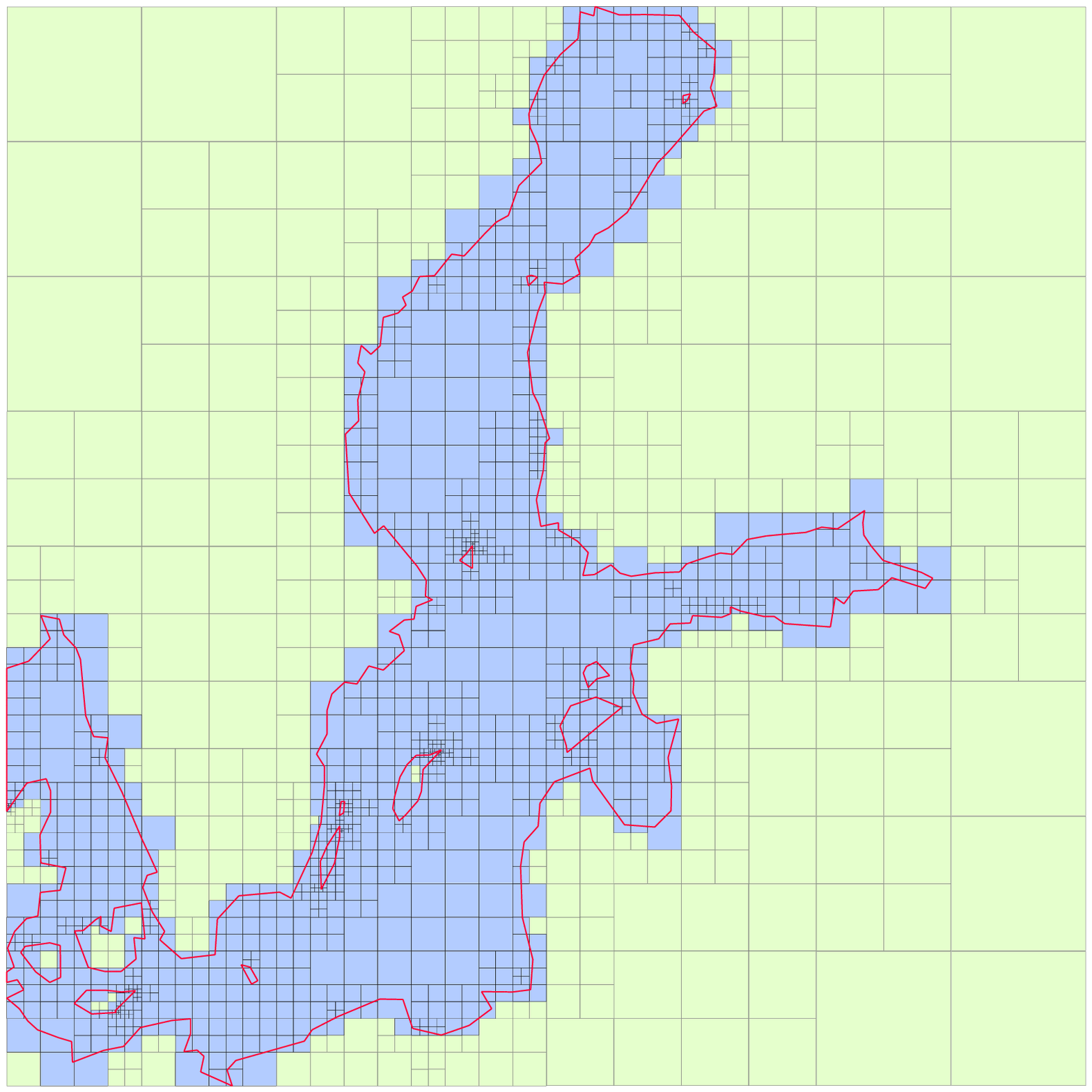}
  \end{center}
  \caption{The finest auxiliary grid $\sigma^{(10)}$ contains
    elements of different size. Left: Dirichlet b.c. ($852$ degrees
    of freedom), right: Neumann b.c. ($2100$ degrees of freedom)\label{Fig:boundary_lvl10}}  
\end{figure}

\subsection{Near boundary correction}\label{sec:boundary}
Since the boundaries of different levels do not coincide, the near boundary error cannot be reduced 
very well by the standard multigrid method for the Neumann boundary condition. So we introduce
a near-boundary region $\Omega_{(\ell,j)}$ where a correction for 
the boundary approximation will be done. The near-boundary region 
is defined in layers around the boundary $\Gamma_{\ell}$:

\begin{definition}[Near-boundary region]
  We define the $j$-th near-boundary region $\mathcal{T}_{(\ell,j)}$ 
  on level $\ell$ of the auxiliary grids by 
 $$
 \begin{aligned}
 \mathcal{T}_{(\ell,0)} :=& \{\tau\in\mathcal{T}_{\ell}\mid 
  {\rm dist}(\Gamma_{\ell},\tau)=0\},\\
  \mathcal{T}_{(\ell,i)} :=& \{\tau\in\mathcal{T}_{\ell}\mid 
  {\rm dist}(\mathcal{T}_{(\ell,i-1)},\tau)=0\},
\end{aligned}
  \quad i=1,\ldots,j.
$$
\end{definition}

The idea for solving the linear system on level 
$\ell$ is to perform a near-boundary correction after the coarse grid correction. The errors introduced by the
coarse grid correction is eliminated by solving the subsystem for the degrees of freedom in the near-boundary 
region $\mathcal{T}_{(\ell,j)}$.
%, where for the Dirichlet problem $j=5$ seems to be adequate. For fewer layers the error by the boundary layers would be dominant whereas for more layers the complexity would grow fast. 
The extra computational complexity is $\mathcal{O}(N)$ because only the elements which are close to the boundary are considered.

\begin{definition}[Partition of degrees of freedom]
  Let ${\jndex}_{\ell}$ denote the index set for the 
  degrees of freedom on the auxiliary grid $\mathcal{T}_{\ell}$.
  We define the near-boundary degrees of freedom by
  \[{\jndex}_{\ell,j} := \{i\in{\jndex}_{\ell}\mid 
  i \mbox{ belongs to an element }\tau\in\mathcal{T}_{(\ell,j)}\}.
  \]
\end{definition}

Let $(u)_{i\in{\jndex}_{\ell}}$ be a coefficient vector on
level $\ell$ of the auxiliary grids. 
Then we extend the standard coarse grid correction by the solve step 
\[
r^{\ell} := f - A^{\ell}u^\ell,\qquad
u^{\ell}|_{\jndex_{\ell,j}} := 
(A^{\ell}|_{\jndex_{\ell,j}\times\jndex_{\ell,j}})^{-1}
r^{\ell}|_{\jndex_{\ell,j}}.
\]
The small system $A^{\ell}|_{\jndex_{\ell,j}}$ of near-boundary
elements is solved by an $\cal H$-matrix solver, cf. \cite{GRKRLE05}.
\begin{algorithm}[htb]
  \caption{Auxiliary Space MultiGrid\label{Alg:asmg1}}
 \begin{algorithmic}
\STATE For $\ell=0$, define $B_0=A_0^{-1}$.  Assume that $B_{\ell-1}: V_{\ell-1}\rightarrow V_{\ell-1}$ is defined.  We shall now define $B_{\ell}: V_{\ell}\rightarrow V_{\ell}$ which is an iterator for the equation of the form $$A_\ell u=f.$$
\STATE {\bf Pre-smoothing:} For $u^0=0$ and $k=1,2,\cdots, \nu$
        $$u^k=u^{k-1}+R_{\ell}(f-A_\ell u^{k-1})$$
\STATE {\bf Coarse grid correction:} $e_{\ell-1}\in V_{\ell-1}$ is the approximate solution of the residual equation $A_{\ell-1}e=Q_{\ell-1}(f-A_\ell u^{\nu})$ by the iterator $B_{\ell-1}$:
$$u^{\nu+1}=u^{\nu}+e_{\ell-1}=u^{\nu}+B_{\ell-1}Q_{\ell-1}(g-A_\ell u^{\nu}).$$
\STATE \textcolor{blue}{{\bf Near boundary correction:} 
$$u^{\nu+2} = u^{\nu+1} + u^{\ell}|_{\jndex_{\ell,j}}= u^{\nu+1} + 
  \left(A^{\ell}|_{\jndex_{\ell,j}\times\jndex_{\ell,j}}\right)^{-1}(f -A_\ell u^{\nu+1} )$$.
 }
\STATE {\bf Post-smoothing:} For $k=\nu+3,\cdots, 2\nu+3$
$$u^{k}=u^{k-1}+R_{\ell}(f-A_{\ell}u^{k-1})$$
\end{algorithmic}
\end{algorithm}

\section{Convergence of the auxiliary grid method}\label{sec:proofaux}

In this section, we investigate and analyze the new algorithm by verifying the assumptions of the theorem of the auxiliary grid method. 
\subsection{Overall algorithm}

Based on the auxiliary hierarchy we constructed in Section \ref{sec:construct}, we can define the auxiliary space preconditioner 
(\ref{eq:fasp}) and (\ref{eq:fasp2}) as follows. 

Let the auxiliary space $V = V_{J}$ and $\tilde{A}$ be generated from (\ref{eq:bilinearform}).  Since we already have the hierarchy of grids $\{V_{\ell}\}_{\ell=1}^{J}$, we can apply MG on the auxiliary space $V_{J}$ as the preconditioner $\tilde{B}$. On the space $\mathcal{V}$, we can apply a traditional smoother $S$,
e.g. Richardson, Jacobi, or Gau\ss{}-Seidel.
For the stiffness matrix $A = D-L-U$ (diagonal, lower and upper triangular part), 
the matrix representation of the Jacobi iteration is $S=D^{-1}$ and for the 
Gau\ss{}-Seidel iteration it is $ S = (D-L)^{-1}$.
(More generally, one could use any smoother that features the spectral equivalence~
$\|v\|_{S^{-1}} \ec \parallel h^{-1}v\parallel^{2}_{L^2(\Omega)}$.)

The auxiliary grid may be over-refined, it can happen that an element 
$\tau_i\in\mathcal{T}$ intersects much smaller auxiliary elements $\tau^{J}_{j}\in\mathcal{T}_{J}$:
\begin{equation}\label{eq:Adens2}
\tau_i\cap\tau^{J}_{j}\ne\emptyset,\ 
h_{\tau^{J}_{j}}\lesssim h_{\tau_i}\quad\text{but}\quad 
h_{\tau^{J}_{j}}\not\ec h_{\tau_i}.
\end{equation}
In this case, we do not have the local approximation and stability properties for the standard nodal interpolation
operator. Therefore, we need a stronger interpolation between the original space and the auxiliary space. 
This is accomplished by the Scott-Zhang quasi-interpolation operator 
\[
\Pi: H^1(\Omega)\rightarrow \mathcal{V}
\]
for a triangulation $\mathcal{T}$ \cite{scott.Zhang.1990}. Let $\{\psi_i\}$ be an $L^{2}$-dual basis to the 
nodal basis $\{\varphi_i\}$. We define the interpolation operator as 
\[
\Pi v(x) := \sum_{i\in\jndex}\varphi_i(x)\int_{\Omega}\psi(\xi)v(\xi)d\xi.
\]
By definition, $\Pi$ preserves piecewise linear functions and satisfies (cf. \cite{scott.Zhang.1990}) for all $v\in H^1(\Omega)$
\begin{equation}\label{Prop1:Zhang}
|\Pi v|^{2}_{1,\Omega} + \sum_{\tau\in\mathcal{T}}h_{\tau}^{-2}\|(v-\Pi v)\|^{2}_{0,\tau}\lesssim |v|^{2}_{1,\Omega}.
\end{equation}
We define the new interpolation $\Pi$ from the auxiliary space $V$ to $\mathcal{V}$ by the Scott-Zhang interpolation
$\Pi: V\rightarrow \mathcal{V}$ and the reverse interpolation 
$\tilde\Pi: \mathcal{V}\rightarrow V$.

Then we can apply Theorem \ref{thm:FASP} for $V=V_{J}$. In order to estimate the condition number, we need to verify that the multigrid preconditioner $\tilde{B}$ on the auxiliary space is  bounded and the finest auxiliary grid and
  corresponding FE space we constructed yields a stable and bounded
  transfer operator and smoothing operator.

\subsection{Convergence of the MG on the auxiliary grids}\label{sec:proofmg}

Firstly, we prove the convergence of the multigrid method on the auxiliary space. 
For the Dirichlet boundary, we have the nestedness
\[
\Omega^{D}_1\subset \cdots \subset\Omega^{D}_J\qquad\subset\Omega,
\]
which induces the nestedness of the finite element spaces defined on the auxiliary grids
$\mathcal{T}^{D}_{\ell},\ell=1,\cdots,J$: 
\[
V_1\subset V_2\subset\cdots\subset V_{J}.
\]
In order to avoid overloading the notation, we will skip the superscript $D$ in the following.

In order to prove the convergence of the local multilevel methods by Theorem \ref{thm:ssc}, we only need to verify the assumptions for the decompsition of 
\begin{equation}\label{eq:decomposition}
V_{J} = \sum_{\ell = 1}^{J}\sum_{k\in\tilde{N}_{\ell}}V_{\ell,k}.
\end{equation}
 where 
 $$
 \tilde{N}_{\ell}=\{k\in \jndex_{\ell}|k\in \jndex_{\ell}\setminus \jndex_{\ell-1}\text{ or } \varphi_{k,\ell}\neq \varphi_{k,\ell-1}\}.
 $$
Since $\mathcal{T}_{\ell}\subset\sigma^{(\ell)}$ is the local 
refinement of $\mathcal{T}_{\ell-1}$, the size of the  triangles in $\mathcal{T}_{\ell}$ may be different.
%The triangulations have to be modified for the analysis, afterwards we consider the given case of locally refined grids. 
We denote $\bar{\cal T}_\ell$ as a
refinement of the grid ${\cal T}_\ell$ where all elements are regularly refined such that 
all elements from $\bar{\cal T}_\ell$ are congruent to the smallest element of ${\cal T}_\ell$. The finite element spaces corresponding to $\bar{\cal T}_\ell$ are denoted by $\bar{V}_\ell$. 
In the triangulations $\bar{\mathcal{T}}_\ell$ we have
\[\tau\in \bar{\mathcal{T}}_\ell \quad\Rightarrow\quad h_{\tau}\sim 2^{-\ell}.\]
For an element $\tau\in\mathcal{T}_\ell$ we denote by $g_{\tau}$ the level number 
of the triangulation $\bar{\mathcal{T}}_{g_{\tau}}$ to which $\tau$ belongs, 
i.e. $h_{\tau}\sim 2^{-g_{\tau}}$. For any vertex $p_{i}$, if $i\in\jndex_{\ell}$ but $i\not\in\jndex_{\ell-1}$, we define $g_{p_{i}} = \ell$.
The following properties about the generation of elements or vertices are \cite{Xu.J.Chen.L.Nochetto.R2009,xu.chen.nochetto.2012}
\begin{align*}
\tau\in \bar{\mathcal{T}}_{\ell}, \text{ if and only if }g_{\tau} = \ell;\\
i\in \jndex_{\ell}, \text{ if and only if }g_{p_{i}}\leq \ell;\\
\text{For } \tau\in\bar{\mathcal{T}}_{\ell}, \max_{i\in\jndex(\tau)}g_{p_{i}}=\ell=g_{\tau},
\end{align*}
where $\jndex(\tau)$ is the set of vertices of $\tau\in\bar{\mathcal{T}}_{\ell}$.

With the space decomposition (\ref{eq:decomposition}), we can verify the assumptions of Theorem \ref{thm:ssc}.

\subsubsection{Stable decomposition: Proof of {\bf(A1)}} The purpose of this subsection is to prove the decomposition is stable.
\begin{theorem}\label{thm:A1}
For any $v\in V$, there exist function $v^{\ell}_{i}\in V_{\ell,i},i\in\tilde{N}_{\ell}$, $\ell=1,\ldots,J$, such that 
\begin{equation}
v =  \sum_{\ell=1}^{J}\sum_{i\in\tilde{N}_{\ell}}v^{\ell}_{i} 
\quad\text{and}\quad
\sum_{\ell=1}^{J}\sum_{i\in\tilde{N}_{\ell}}\|v^{\ell}_{i}\|^{2}_{A}\lesssim \log(N)\|v\|^{2}_{A}.
\end{equation}
\end{theorem}
\begin{proof}
Following the argument of \cite{Xu.J.Chen.L.Nochetto.R2009,xu.chen.nochetto.2012}, we define the Scott-Zhang interpolation between different levels $\Pi_{\ell}:V_{\ell+1}\rightarrow V_{\ell}$, $\Pi_{L}:V_{L}\rightarrow V_{L}$, and $\Pi_{0}:V_{1}\rightarrow 0$.

By the definition, we can define the decomposition as
 \[
 v = \sum_{\ell=1}^{J}v^{\ell},\quad v^{\ell}=(\Pi_{\ell}-\Pi_{\ell-1})v \in V_{\ell}.
 \]
 Assume $v^{\ell} = \sum_{i\in\jndex_{\ell}}\xi_{\ell,i}\varphi^{\ell}_{i}$, where $v^{\ell}_{i} = \xi_{\ell,i}\varphi^{\ell}_{i}\in V_{\ell,i}$. Then,
\[
\|v^{\ell}_{i}\|^{2}_{0} 
= \|v^{\ell}_{i}\|^{2}_{0,\omega^{\ell}_{i}}
\lesssim \sum_{\tau\in\omega^{\ell}_{i}}h^{d}_{\tau}|v^{\ell}(p_{i})|^{2}
\lesssim\|v^{\ell}\|^{2}_{0,\omega^{\ell}_{i}}
=\|(\Pi_{\ell}-\Pi_{\ell-1})v\|^{2}_{0,\omega^{\ell}_{i}}.
\]
where $\omega^{\ell}_{i}$ is support of $\varphi^{\ell}_{i}$ and the center vertex is $p_{i}$.

By the inverse inequality, we can conclude
\[
\sum_{i\in\jndex_{\ell}}\|v^{\ell}_{i}\|^{2}_{A}\lesssim\sum_{i\in\jndex_{\ell}}\sum_{\tau\in\omega^{\ell}_{i}}h_{\tau}^{-2}\|v^{\ell}_{i}\|^{2}_{0,\tau}\lesssim\sum_{\tau\in\mathcal{T}_{\ell}}h_{\tau}^{-2}\|(\Pi_{\ell}-\Pi_{\ell-1})v\|^{2}_{0,\tau}.
\]
%\[
%\sum_{\ell=1}^{p}\sum_{i\in\jndex_{\ell}}\|v^{\ell}_{i}\|^{2}_{A}\lesssim\sum_{\ell=1}^{p}\sum_{i\in\jndex_{\ell}} \|h^{-1}v^{\ell}_{i}\|^{2}_{0}\lesssim \sum_{\ell=1}^{p}\|h^{-1}(\Pi_{\ell}-\Pi_{\ell-1})v\|^{2}_{\Omega_{\ell}}.
%\]
Invoking the approximability and stability and following the same argument of Lemma \ref{lemma:approx1}, we have
$$
  \sum_{\tau\in\mathcal{T}_{\ell}}h_{\tau}^{-2}\|(v-\Pi_{\ell} v)\|^{2}_{0,\tau}\lesssim |v|^{2}_{1,\Omega_{\ell+1}}\quad\text{and}\quad\|\Pi_{\ell} v\|^{2}_{0,\Omega_{\ell}}\lesssim \|v\|^{2}_{0,\Omega_{\ell+1}}.
$$
So,
\begin{align*}
\sum_{\tau\in\mathcal{T}_{\ell}}h_{\tau}^{-2}\|(\Pi_{\ell}-\Pi_{\ell-1})v\|^{2}_{0,\tau}
&=\sum_{\tau\in\mathcal{T}_{\ell}}h_{\tau}^{-2}\|\Pi_{\ell}(I - \Pi_{\ell-1})v\|^{2}_{0,\tau}\\
&\lesssim \sum_{\tau\in\mathcal{T}_{\ell}}h_{\tau}^{-2}\|(I - \Pi_{\ell-1})v\|^{2}_{0,\omega^{\ell}_{\tau}}
\lesssim \sum_{\tau\in\mathcal{T}_{\ell}}h_{\tau}^{-2}\|(I - \Pi_{\ell-1})v\|^{2}_{0,\tilde{\omega}^{\ell}}\\
&\lesssim \sum_{\tau\in\mathcal{T}_{\ell-1}}h_{\tau}^{-2}\|(I - \Pi_{\ell-1})v\|^{2}_{0,\tau}
\lesssim |v|^{2}_{1}.
\end{align*}
where $\omega^{\ell}_{\tau}$ is the union of the elements in $\mathcal{\ell}$ that intersect with $\tau\in\mathcal{T}_{\ell}$ and $\tilde{\omega}^{\ell}_{\tau}$ is the union of the elements in  $\mathcal{T}_{\ell-1}$ that intersect with $\omega^{\ell}_{\tau}$.
Therefore,
$$
\sum_{\ell=1}^{J}\sum_{i\in\tilde{N}_{\ell}}\|v^{\ell}_{i}\|^{2}_{A}\lesssim\sum_{\ell=1}^{J}\sum_{\tau\in\mathcal{T}_{\ell}}h_{\tau}^{-2}\|(\Pi_{\ell}-\Pi_{\ell-1})v\|^{2}_{0,\tau}\lesssim J|v|^{2}_{1} \lesssim\log(N)|v|^{2}_{1}.
$$
\end{proof}

\subsubsection{Strengthened Cauchy-Schwarz inequality: Proof of {\bf (A2)}}
In this subsection, we establish the strengthened Cauchy-Schwarz inequality for the space decomposition (\ref{eq:decomposition}). Assuming there is an ordering index set $\Lambda = \{\alpha|\alpha=(\ell_{\alpha},k_{\alpha}), k_{\alpha}\in \tilde{N}_{\ell},\ell_{\alpha} = 1,\cdots,J\}$. 
Define the ordering as follows. For any $\alpha,\beta\in\Lambda$, if $\ell_{\alpha}> \ell_{\beta}$ or $\ell_{\alpha} = \ell_{\beta}, k_{\alpha}>k_{\beta}$, 
then, $\alpha>\beta$. The strengthened Cauchy-Schwarz inequality is given as follows.

\begin{theorem}\label{SCS1}
For any $u_{\alpha}=v^{\ell}_{k}\in V_{\alpha} = V_{\ell,i},v_{\beta}=v^{m}_{j}\in V_{\beta} = V_{m,j},\alpha=(\ell,k),\beta=(m,j)\in\Lambda$, we have 
\[
\bigg|\sum_{\alpha\in\Lambda}\sum_{\beta\in\Lambda,\beta>\alpha}(u_{\alpha},v_{\beta})_A\bigg|\lesssim\bigg(\sum_{\alpha\in\Lambda}\|u_{\alpha}\|^2_A\bigg)^{1/2}\bigg(\sum_{\beta\in\Lambda}\|u_{\beta}\|^2_A\bigg)^{1/2}.
\]
\end{theorem}
In order to prove the theorem, we need the following lemma:
\begin{lemma}[SCS inequality for quasi-uniform meshes]
For any $u_{i}\in \bar{V}_{i},v_{j}\in\bar{V}_{j}$, we have
\[
(u_{i},v_{j})_{1}\lesssim \left(\frac{h_{j}}{h_{i}}\right)^{1/2}|u_{i}|_{1}(h_{j}^{-1}\|v_{j}\|_{0}).
\]
\end{lemma}
The proof of the lemma follows from Lemma 4.26 in \cite{Xu1997} and Lemma 4.5 in \cite{Xu.J.Chen.L.Nochetto.R2009}.

Now we can prove the theorem \ref{SCS1}.

\begin{proof}
For any $\alpha\in\Lambda$, we denote by
\[
n(\alpha)=\{\beta\in\Lambda|\beta>\alpha, \omega_{\beta}\cap\omega_{\alpha}\neq\emptyset\},\quad v_k^{\alpha} = \sum_{\beta\in n(\alpha),g_{\beta}=k}v_{\beta}.
\]
where $\omega_{\alpha}$ is the support of the $V_{\alpha}$ and $g_{\alpha} = \max_{\tau\in\omega_{\alpha}} g_{\tau}$.
% the center grid  vertex of the subspace is $p_{\alpha}$ and $g_{\alpha} = g_{p_{\alpha}}$.

Since, the mesh is a K-mesh, for any $\tau\subset \omega_{\alpha}$, we have
\[
(u_{\alpha},v^{\alpha}_k)_{1,\tau}\lesssim \left(\frac{h_{k}}{h_{g_{\alpha}}}\right)^{1/2}|u_{\alpha}|_{1,\tau}h^{-1}_{k}\|v^{\alpha}_k\|_{0,\tau}
\]
So,
\begin{align*}
(u_i,v^{\alpha}_k)_{1,\omega_{\alpha}}&=\sum_{\tau\subset\omega_{i}}(u_i,v^{\alpha}_k)_{1,\tau}\\
&\lesssim\sum_{\tau\subset\omega_{\alpha}}\left(\frac{h_{k}}{h_{g_{\alpha}}}\right)^{1/2}|u_{\alpha}|_{1,\tau}h^{-1}_{k}\|v^{\alpha}_k\|_{0,\tau}\\
&\leq \left(\frac{h_{k}}{h_{g_{\alpha}}}\right)^{1/2}|u_{\alpha}|_{1,\omega_{\alpha}}h^{-1}_{k}\left(\sum_{\beta\in n(\alpha),g_{\beta}=k}\|v_{\beta}\|^{2}_{0,\omega_{\alpha}}\right)^{1/2}.
\end{align*}
Then fix $u_{\alpha}$ and consider
\begin{align*}
\left|(u_{\alpha},\sum_{\beta\in\Lambda,\beta>\alpha}v_{\beta})_{A}\right|&\ec\left|(u_{\alpha},\sum_{\beta\in n(\alpha)}v_{\beta})_{1,\omega_{\alpha}}\right| =\left|(u_{\alpha},\sum_{k=g_{\alpha}}^{J}\sum_{\beta\in n(\alpha)}v_{\beta})_{1,\omega_{\alpha}}\right|\\
&\lesssim \sum_{k=g_{\alpha}}^{J}|(u_{\alpha},v^{\alpha}_{k})_{1,\omega_{\alpha}}|\\
&\lesssim\sum_{k=g_{\alpha}}^{J}\left(\frac{h_{k}}{h_{g_{\alpha}}}\right)^{1/2}|u_{\alpha}|_{1,\omega_{\alpha}}h^{-1}_{k}\bigg(\sum_{g_{\beta}=k}\|v_{\beta}\|^{2}_{0,\omega_{\alpha}}\bigg)^{1/2}\\
\end{align*}

We sum up the $u_{\alpha}$ level by level,
{\small
\begin{align*}
\sum_{\ell=1}^{J}\sum_{g_{\alpha}=\ell}\left|(u_{\alpha},\sum_{\scriptstyle\beta\in\Lambda\atop\scriptstyle\beta>\alpha}v_{\beta})_{A}\right|&\lesssim\sum_{\ell=1}^{J}\sum_{g_{\alpha}=\ell}\left[\sum_{k=\ell}^{J}\left(\frac{h_{k}}{h_{\ell}}\right)^{\frac{1}{2}}|u_{\alpha}|_{1,\omega_{\alpha}}h^{-1}_{k}\bigg(\sum_{\scriptstyle\beta\in n(\alpha)\atop\scriptstyle g_{j}=k}\|v_{j}\|^{2}_{0,\omega_{\alpha}}\bigg)^{\frac{1}{2}}\right]\\
&\lesssim\sum_{\ell=1}^{J}\sum_{k=\ell}^{J}\Bigg(\frac{h_{k}}{h_{\ell}}\sum_{g_{\alpha}=\ell}|u_{\alpha}|^{2}_{1,\omega_{\alpha}}\Bigg)^{\frac{1}{2}}\Bigg(h^{-2}_{k}\sum_{g_{\alpha}=\ell}\sum_{\scriptstyle\beta\in n(\alpha)\atop\scriptstyle g_{\beta}=k}\|v_{j}\|^{2}_{0,\omega_{\alpha}}\Bigg)^{\frac{1}{2}}\\
&\lesssim\sum_{\ell=1}^{J}\sum_{k=\ell}^{J}\Bigg(\frac{h_{k}}{h_{\ell}}\Bigg)^{\frac{1}{2}}\Bigg(\sum_{g_{\alpha}=\ell}|u_{\alpha}|^{2}_{1,\omega_{\alpha}}\Bigg)^{\frac{1}{2}}\Bigg(\sum_{g_{\alpha}=\ell}\sum_{g_{\beta}=k}\frac{h^{2}_{\ell}}{h^{2}_{k}}|v_{\beta}|_{1,\omega_{\alpha}}^{2}\Bigg)^{\frac{1}{2}}\\
&\lesssim\Bigg(\sum_{\ell=1}^{J}\sum_{g_{\alpha}=\ell}|u_{\alpha}|^{2}_{1}\Bigg)^{\frac{1}{2}}\Bigg(\sum_{\ell=1}^{J}\sum_{k=\ell}^{J}\sum_{g_{\alpha}=\ell}\sum_{g_{\beta}=k}\frac{h^{2}_{\ell}}{h^{2}_{k}}|v_{\beta}|_{1,\omega_{\alpha}}^{2}\Bigg)^{\frac{1}{2}}\\
&\lesssim\Bigg(\sum_{\ell=1}^{J}\sum_{g_{i}=\ell}\|u_{i}\|^{2}_{A}\Bigg)^{\frac{1}{2}}\Bigg(\sum_{k=1}^{J}\sum_{\ell=1}^{k}\frac{h^{2}_{\ell}}{h^{2}_{k}}\sum_{g_{\beta}=k}|v_{\beta}|_{1}^{2}\Bigg)^{\frac{1}{2}}.\\
&\lesssim\Bigg(\sum_{\ell=1}^{J}\sum_{g_{i}=\ell}\|u_{i}\|^{2}_{A}\Bigg)^{\frac{1}{2}}\Bigg(\sum_{k=1}^{J}\sum_{g_{\beta}=k}\|v_{\beta}\|_{A}^{2}\Bigg)^{\frac{1}{2}}.\\
%&\lesssim(\sum_{i\in\Lambda}\|u_{i}\|^{2}_{A})^{1/2}(\sum_{j\in\Lambda}\|v_{j}\|_{A}^{2})^{1/2}.\\
\end{align*}
}
This gives us the desired estimate.
\end{proof}

The Gau\ss-Seidel method as the smoother means choosing the exact inverse for each of the subspaces $V_{\ell,k}$.
Therefore, the assumption of the smoother is satisfied as well. Consequently, we have the uniform convergence 
of the multigrid method on the auxiliary grid.
\begin{theorem}
The multigrid method on the auxiliary grid based on the space decomposition (\ref{eq:decomposition})
is nearly optimal, the convergence rate is bounded by $1 - \frac{1}{1+C\log(N)}$.
\end{theorem}

\subsubsection{Condition number estimation}
Now, we estimate condition number of the auxiliary space preconditioner by verifying the assumptions in Theorem \ref{thm:FASP}.

The assumption (\ref{eqn:smoother})  is the continuity of the smoother $S$. We prove the first assumption in Theorem \ref{thm:FASP} for the Jacobi and Gau\ss{}-Seidel iteration.
For the Jacobi method, the square of the energy norm can be computed by summing 
local contributions from the cells $\tau_i$ of the mesh $\mathcal{T}$:
\begin{align*}
\| v\|^{2}_{A} &= \|\sum_{i\in\jndex}\xi_i\varphi_i\|^{2}_{A}= a(\sum_{i\in\jndex}\xi_i\varphi_i,\sum_{j\in\jndex}\xi_j\varphi_j)\quad
= \sum_{i\in\jndex}\sum_{\omega_i\cap\omega_j\ne\emptyset}a(\xi_i\varphi_i,\xi_j\varphi_j)\\
&\leq \sum_{i\in\jndex}\sum_{\omega_i\cap\omega_j\ne\emptyset} \frac{1}{2}(\|\xi_i\varphi_i\|^2_A+\|\xi_j\varphi_j\|^2_A)\quad
\leq K\sum_{i\in\jndex}\|\xi_i\varphi_i\|^{2}_{A}\quad
= K\langle Dv,v\rangle,
\end{align*}
where K is the maximal number of non-zeros in a row of A. Thus the choice $c_s=K$ fulfills the continuity assumption. The continuity of the Gau\ss{}-Seidel method can also be proved.

\begin{lemma}[Continuity for Gau\ss{}-Seidel]
  The stiffness matrix $A = D-L-U$ fulfills
  \begin{equation}
    \frac{1}{K}\langle(D-L)\xi,\xi\rangle \le \langle D\xi,\xi\rangle \le 2\langle(D-L)\xi,\xi\rangle,\quad \xi\in \mathbb{R}^N.
  \end{equation}
\end{lemma}

In order to prove assumptions (\ref{eqn:transfer}) and (\ref{eqn:transfer_stable}) , we need the  following lemmas for the transfer operator between $\mathcal{V}$  and $V$.

\begin{lemma}[Local stability property]\label{Lem:stab}
  For any auxiliary space function $v\in \tilde{V}$ and any element $\tau\in\mathcal{T}$, 
  the quasi-interpolation $\Pi$ satisfies
  \[|\Pi v|_{k,\tau}\lesssim h_{\tau}^{j-k}|v|_{j,\omega_{\tau}},\qquad
  j,k\in\{0,1\},\]
  where $\omega_{\tau}$ is the union of elements in the auxiliary grid $\mathcal{T}$ that intersect with $\tau$.
\end{lemma}

\begin{lemma}\label{lemma:approx1}
For any auxiliary element function $v\in V$ the reverse interpolation operator $\tilde\Pi$  satisfies
\begin{equation}\label{SZ:approx}
  \sum_{\tau\in\mathcal{T}}h_{\tau}^{-2}\|(v-\tilde\Pi v)\|^{2}_{0,\tau}\lesssim |v|^{2}_{1,\Omega}\quad\text{and}\quad  |\tilde\Pi v|^{2}_{1,\Omega_{J}}\lesssim |v|^{2}_{1,\Omega}.
\end{equation}
\end{lemma}
\begin{proof}
The proof follows an argument presented in Xu \cite{Xu1996a}. Let $\hat{\mathcal{T}}$ be the set of the elements in $\mathcal{T}$ which do not intersect with $\partial\Omega_{J}$, i.e.
$$
\hat{\mathcal{T}}=\{\tau|\tau\in\mathcal{T},\tau\in\Omega_{J}\tau\cap\partial\Omega_{J}=\emptyset\},\quad
\overline{\hat{\Omega}} = \bigcup_{\tau\in\hat{\mathcal{T}}}\bar{\tau}.
$$
Then,
$$
\sum_{\tau\in\mathcal{T}}h_{\tau}^{-2}\|(v-\tilde\Pi v)\|^{2}_{0,\tau}\leq
\sum_{\tau\in\hat{\mathcal{T}}}h_{\tau}^{-2}\|(v-\tilde\Pi v)\|^{2}_{0,\tau}
+\sum_{\tau\in\mathcal{T}\setminus\hat{\mathcal{T}}}h_{\tau}^{-2}\|v\|^{2}_{0,\tau}
+\sum_{\tau\in\mathcal{T}\setminus\hat{\mathcal{T}}}h_{\tau}^{-2}\|\tilde\Pi v\|^{2}_{0,\tau}.
$$
For any element $\tau\in\hat{\mathcal{T}}$, $\omega_{\tau}$ is the union of elements in the auxiliary grid $\mathcal{T}_{J}$ that intersect with $\tau$,% and $\sigma_{\tau}\subset\Omega$ is the local patch covering $\omega_{\tau}$,
\begin{equation}\label{eq:stat1}
h_{\tau}^{-2}\|(v-\tilde\Pi v)\|^{2}_{0,\tau}
\lesssim \sum_{\tilde{\tau}\subset \omega_{\tau}}h_{\tilde{\tau}}^{-2}\|(v-\tilde\Pi v)\|^{2}_{0,\tilde{\tau}}
\lesssim |v|^{2}_{1,\omega_{\tau}}
%\lesssim\sum_{\tau_{j}\subset \sigma_{\tau}}|v|^{2}_{1,\tau_{j}}.
\end{equation}
%\lesssim \sum_{\bar{\tau}_{j}\subset \omega_{\tau}}h^{-2}_{\bar{\tau}_{j}}\|v-\tilde\Pi v\|^{2}_{0,\bar{\tau}_{j}}
So,
$$
\sum_{\tau\in\hat{\mathcal{T}}}h_{\tau}^{-2}\|(v-\tilde\Pi v)\|^{2}_{0,\tau}
\lesssim\sum_{\tau\in\hat{\mathcal{T}}}|v|^{2}_{1,\omega_{\tau}}
\lesssim|v|^{2}_{1,\Omega_{J}}\leq|v|^{2}_{1,\Omega}.
$$

By the Poincar\'{e} inequality and scaling, if $G^{\eta}$ is a reference square ($d=2$) or a cube ($d=3$) of side length $\eta$, then
$$
\eta^{-2}\|w\|^{2}_{0,G^{\eta}}\lesssim
\int_{G^{\eta}}|\nabla w|^{2}dx
$$
holds for all functions $w$ vanishing on one edge of $G^{\eta}$. 
%By covering the strip $\Omega\setminus\hat{\Omega}$ with subregions each of which can be mapped onto $G^{\eta}$, we can deduce that 
For any $\tau\in\mathcal{T}\setminus\hat{\mathcal{T}}$, by covering $\tau$ with subregion which can be mapped onto $G^{\eta_{\tau}}, \eta_{\tau}\ec h_{\tau}$, we can conclude that
$$
\sum_{\tau\in\mathcal{T}\setminus\hat{\mathcal{T}}}h_{\tau}^{-2}\|w\|^{2}_{0,\tau}
\lesssim \sum_{\tau\in\mathcal{T}\setminus\hat{\mathcal{T}}}|w|^{2}_{1,G^{\eta_{\tau}}}\lesssim|w|^{2}_{1,\Omega}.
$$
Applying the above estimate with $w=v$ and $w=\tilde\Pi v$, one has
$$
\sum_{\tau\in\mathcal{T}\setminus\hat{\mathcal{T}}}h_{\tau}^{-2}\|v\|^{2}_{0,\tau}
+\sum_{\tau\in\mathcal{T}\setminus\hat{\mathcal{T}}}h_{\tau}^{-2}\|\tilde\Pi v\|^{2}_{0,\tau}
\lesssim |v|^{2}_{1,\Omega}
+|\tilde\Pi v|^{2}_{1,\Omega}\lesssim|v|^{2}_{1,\Omega}.
$$
For the second inequality, 
$$
|\tilde\Pi v|^{2}_{1,\Omega_{J}}\lesssim |v|^{2}_{1,\Omega_{J}}\lesssim |v|^{2}_{1,\Omega}
$$
So, we have the desired estimate.
\end{proof}

We can now verify the remaining assumptions of 
Theorem \ref{thm:FASP}.
\begin{lemma}
  For any $v\in V_{p}$, we have
  \[
  |\Pi v|_{1,\Omega} \lesssim |v|_{1,\Omega_{J}}.
  \]
\end{lemma}
\begin{proof}
  By the local stability of $\Pi$, 
\[
    |\Pi v|^{2}_{1,\Omega} 
 = \sum_{\tau\in\mathcal{T}}|\Pi v|^{2}_{1,\tau} \lesssim \sum_{\tau\in\mathcal{T}}|v|^{2}_{1,\omega_{\tau}}\lesssim|v|^{2}_{1,\Omega} =|v|^{2}_{1,\Omega_{J}}
\]
  The desired estimate then follows.
\end{proof}

\begin{lemma}
For any $v\in V$, there exists $v_0\in V$ and $w\in V_p$ such that 
\[
\|v_0\|^{2}_{S^{-1}} + |w|^2_{1,\Omega_{p}}\lesssim|v|^2_{1,\Omega}.
\]
\end{lemma}
\begin{proof}
For any $v\in V$, let $w := \tilde\Pi v$ and $v_0 = v - \Pi w$, then
\begin{align*}
  \|v_0\|^{2}_{S^{-1}} + |w|^2_{1,\Omega_{J}} 
  & \lesssim\sum_{\tau\in\mathcal{T}}h_{\tau}^{-2}\|v - \Pi w\|^2_{0,\tau} + |w|^2_{1,\Omega_{J}}\\
  & \leq \sum_{\tau\in\mathcal{T}}h_{\tau}^{-2}\|v - \tilde\Pi v\|^2_{0,\tau} + \sum_{\tau\in\mathcal{T}}h_{\tau}^{-2}\|w - \Pi w\|^2_{0,\tau} +|w|^2_{1,\Omega_{J}}\\
  &\lesssim|v|^2_{1,\Omega}.
\end{align*}
\end{proof}

\begin{theorem}
If the multigrid method on the Dirichlet auxiliary grid is the preconditioner $\tilde{B}$ on the auxiliary space and the ASMG preconditioner defined by (\ref{eq:fasp}) or (\ref{eq:fasp2}), then 
\[
\kappa(BA)\lesssim log(N).
\]
\end{theorem}

\section{Numerical results}\label{sec:numerics}

The numerical tests in this section are all performed on a SunFire with a
$2.8$GHz Opteron processor and sufficient main memory. Although the tests 
are done using only a single processor with access to all the memory, there 
might be some undesirable scaling effects when using a large portion of the
available memory due to the speed of memory access. Therefore, the timings 
for larger problems are slightly worse than in theory (memory used and flops
counted).

In order to compare our code we need a reference point,  and this reference is 
a straightforward geometric multigrid method on a structured grid, where
we do not exploit the structure except for it to be a geometric multigird
hierarchy. This means we setup the stiffness matrices as well as prolongation and
restriction matrices as one would do on a general grid hierarchy.

\subsection{Geometric Multigrid}

As a test problem we consider the Poisson equation on the unit square
with homogeneous Dirichlet boundary conditions:
\begin{equation}
  -\Delta u = f \mbox{ in }\Omega:=[0,1]^2,\qquad 
  u = 0 \mbox{ on }\Gamma:=\partial\Omega.
\end{equation}

For this domain, it is straight-forward to construct a
nested hierarchy of regular grids 
$\mathcal{T}_{1},\ldots \mathcal{T}_{J}$ 
and corresponding $P_{1}$ finite element spaces 
$V_1\subset \cdots \subset V_J$ 
with $n_\ell := {\rm dim}(V_\ell) = (2^\ell-1)^2$. 
Let $(\varphi_i^{\ell})_{i=1}^{n_\ell}$ denote a Lagrange basis in 
$V_\ell$. 

The geometric multigrid algorithm is presented in 
Algorithm \ref{Alg:mg}. On each level, a smoothing
iteration is required which we take to be symmetric Gau\ss{}-Seidel.
The timings for the setup of the stiffness matrices $A^{\ell}$ and
the prolongation matrices $P^{\ell}$ on level
$\ell$ as well as for 10 V-cycles ($\nu:=2$ smoothing steps) 
of geometric multigrid are given in Table \ref{t_gmg}.
\begin{table}[htb]
  \caption{The time in seconds for the setup of the matrices
    and for ten steps of V-cycle (geometric) multigrid,
    Algorithm \ref{Alg:mg}.\label{t_gmg}}
  \begin{center}
    \begin{tabular}{l|rrr}
      $\#$dof           & Setup of $A,P$ & 10 V-cycles  
      \\\hline 
      $n_{9}  = 1,050, 625$ & 4.3            & 5.1   \\ 
      $n_{10} = 4, 198, 401$ & 26.7           & 39.7 \\
      $n_{11} = 16, 785, 409$& 108.8          & 152.3
    \end{tabular}
  \end{center}
\end{table}

From the timings, we observe that the geometric multigrid method
(with textbook convergence rates) requires roughly $1$ second
per step per million degrees of freedom, i.e., roughly $5-10$ 
seconds per million degrees of freedom to solve the problem.
For smaller problems cacheing effects seem to speed up the 
calculations.

These results for the geometric multigrid method on a uniform grid 
are now compared with the ASMG method 
%of Algorithm \ref{Alg:asmg}
for the baltic sea mesh with strong local refinement and several 
inclusions. 

\subsection{ASMG for the Dirichlet problem}
Our solver for the unstructured grid from the baltic sea geometry, 
cf. Figure \ref{Fig:baltic}, is a preconditioned conjugate gradient 
method (CG) where the preconditioner is ASMG.
We iterate until the discretisation error on the finest level of the 
hierarchy is met. In particular, a nested iteration from the coarsest to 
the finest level is used to obtain good inital values.

We consider the baltic sea model problem with homogeneous Dirichlet boundary
conditions. The storage complexity and timings are shown in Table \ref{t_asmg_D}. 
%show a similar behavior as for the Neumann case. 
The auxiliary grid hierarchy and matrices require 
roughly 3 times more storage than the given (unstructured) grid, and the ASMG-CG solve takes 
approximately 11 seconds per million degrees of freedom, which is at most two times slower than 
the geometric multigrid method.

\begin{table}[htb]
  \caption{The storage complexity in bytes per degree of freedom 
    (auxiliary grids, auxiliary matrices and $\cal H$-solvers)
    and the solve time in seconds for an ASMG preconditioned 
    cg-iteration.\label{t_asmg_D}}
  \begin{center}
    \begin{tabular}{l|cccc}
      $\#$dof           & aux. storage & storage A  & aux. setup  & ASMG-cg solve (steps) 
      \\\hline 
      $n_{4} = 737,933$  & 509     & 85     & 45.2                 & 12.4  (5) \\ 
      $n_{5} = 2,970,149$ & 351     & 87     & 124                  & 40.2  (5) \\
      $n_{6} = 11,917,397$& 281     & 87     & 414                  & 125.9 (5)\\
      $n_{7} = 47,743,157$& 247     & 88     & 1360                 & 544.9   (5)  
    \end{tabular}
  \end{center}
\end{table}

The convergence rates for the ASMG iteration are given in Figure \ref{Fig:rate_Dirichlet}, 
where we plot the residual reduction factors for the first 50 steps. We observe that the rates
on all levels are uniformly bounded away from 1, roughly of the size $0.4$. 

\begin{figure}[!htb]
  \begin{center}
    \includegraphics[width=7cm,angle=270]{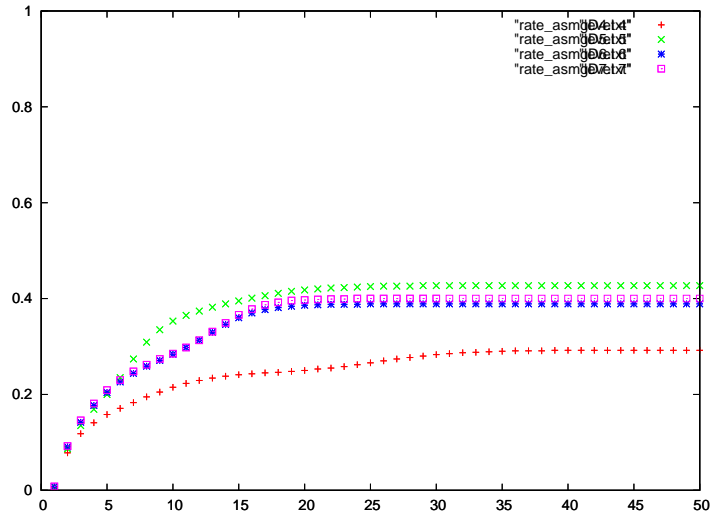}
  \end{center}
  \caption{Covergence rates for Auxiliary Space MultiGrid 
    with $n_{4} = 737,933$, $n_{5} = 2,970,149$, $n_{6} = 11,917,397$, and $n_{7} = 47,743,157$
    degrees of freedom.\label{Fig:rate_Dirichlet}}  
\end{figure}

\subsection{ASMG for the Neumann problem}
In the final test, we consider the baltic sea model problem with homogeneous Neumann boundary conditions. Extra near  boundary correction has been applied, cf. Algorithm \ref{Alg:asmg1}.
The storage complexity as well as the timing for the ASMG-CG solve are given 
in Table \ref{t_asmg_N} for the model problem with natural Neumann boundary
conditions. 

\begin{table}[htb]
  \caption{The storage complexity in bytes per degree of freedom 
    (auxiliary grids, auxiliary matrices and $\cal H$-solvers)
    and the solve time in seconds for an ASMG preconditioned 
    cg-iteration.\label{t_asmg_N}}
  \begin{center}
    \begin{tabular}{l|cccc}
      $\#$dof           & aux. storage  & aux. setup  & storage A  & ASMG-cg solve (steps) 
      \\\hline 
      $n_{4} = 756,317$  & 355           & 34.6        & 88         & 11.5 (6) \\ 
      $n_{5} = 3,006,917$ & 244           & 80          & 88         & 33.3 (6) \\
      $n_{6} = 11,990,933$& 195           & 266         & 88         & 143.6 (6)\\
      $n_{7} = 47,890,229$& 172           & 941         & 88         & 520.7 (6)  
    \end{tabular}
  \end{center}
\end{table}

We observe that in order to solve the model problem up to the size of the 
discretization error, we need  to spend twice as much storage for 
the auxiliary than the given (unstructured) fine grid, and the solve 
time is roughly 11 seconds per million degrees of freedom.

\begin{figure}[!htb]
  \begin{center}
    \includegraphics[width=7cm,angle=270]{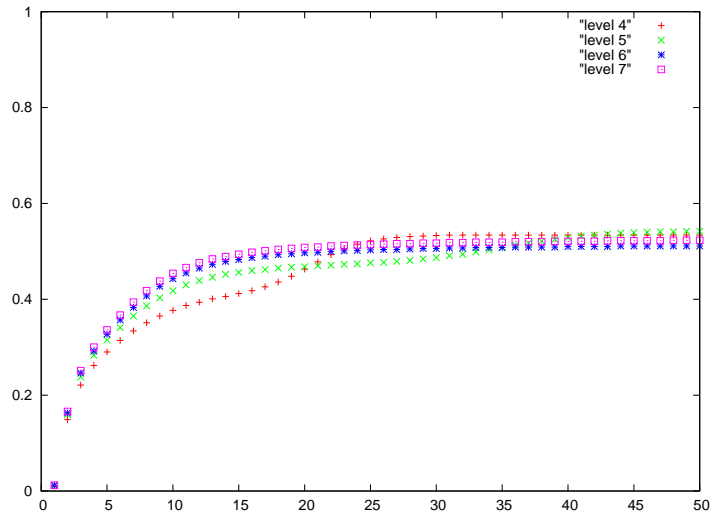}
  \end{center}
  \caption{Covergence rates for Auxiliary Space MultiGrid 
    with $n_{4}=756,317$, $n_{5} = 3,006,917$, $n_{6} = 11,990,933$, and $n_{7} = 47,890,229$
    degrees of freedom.\label{Fig:rate_Neumann}}  
\end{figure}

In Figure \ref{Fig:rate_Neumann} we plot the residual reduction factors for 50 steps 
of ASMG (without CG acceleration) for 4 consecutive levels. We observe that the rate 
is level independent and of $0.5$, i.e. bounded away from 1, but not as good as
the corresponding rate of the geometric multigrid method.

We conclude that the theoretically proven convergence rates are indeed small enough 
to be competitive with geometric multigrid in the sense that the storage requirements increase
by a factor of at most 3 and the solve times at most by a factor of 2. 

%\begin{acknowledgements}
%\end{acknowledgements}

\bibliographystyle{spmpsci}      % mathematics and physical sciences
\bibliography{hmatrix}   % name your BibTeX data base

%\begin{thebibliography}{}
% etc
%\end{thebibliography}

\end{document}